\documentclass[11pt,reqno]{amsart}
\usepackage{amsmath,amssymb,amsfonts}
\usepackage{mathrsfs}
\usepackage{enumerate}
\usepackage[all, knot]{xy}
\usepackage{color}
\usepackage[colorlinks=true, pdfstartview=FitV, linkcolor=blue,%
citecolor=blue, urlcolor=blue]{hyperref}
\usepackage{tikz}
\usepackage{rotating}
\usepackage{graphicx}
\usetikzlibrary{decorations.pathreplacing}
\usepackage{float}
\usepackage{tikz-cd}
\usepackage{ytableau}
\usepackage{mathtools}
\usepackage{appendix}
\newtheorem{Thm}{Theorem}[section]
\newtheorem{Prop}[Thm]{Proposition}
\newtheorem{Cor}[Thm]{Corollary}
\newtheorem{Lem}[Thm]{Lemma}
\theoremstyle{definition}
\newtheorem{defn}[Thm]{Definition}
\newtheorem{Ex}[Thm]{Example}
\newtheorem{Rmk}[Thm]{Remark}
\numberwithin{equation}{section}

\DeclareMathOperator{\evac}{evac}
\DeclareMathOperator{\jdt}{jdt}

\newcommand{\Tinfty}{T_\infty}
\newcommand{\Tinf}{\mathcal T(\infty)}

\newcommand{\f}{\tilde f}

\newcommand{\fp}{\tilde f^{\,\prime}}

\newcommand{\wt}{\mathrm{wt}}

\allowdisplaybreaks
\begin{document}

\title[Young tableau descriptions for the polyhedral realizations]
{Young tableau descriptions for the polyhedral realizations of crystal bases in type $A_n$}

\author[Shaolong Han]{Shaolong Han}
\address{Beijing International Center for Mathematical Research, Peking University, No. 5 Yiheyuan Road, Beijing, 100871, China}
\email{hanshaolong@bicmr.pku.edu.cn}

\subjclass[2020]{17B10, 17B37, 05E10}
\keywords{Crystal basis, Polyhedral realization, Young tableau, Gelfand–Tsetlin pattern, Lusztig data}

\begin{abstract}
By utilizing the combinatorial properties of various tableau models, we establish an explicit correspondence between the polyhedral realizations of the crystal bases \( \mathcal B(\lambda) \) (resp. \( \mathcal B(\infty) \))  of type $A_n$ and the reverse semi-standard Young tableaux (resp. reverse marginally large tableaux), thereby providing a combinatorial description of the corresponding polyhedral realizations. Furthermore, a crystal structure on the set of Gelfand–Tsetlin patterns is obtained via the correspondence between the polyhedral realization of \( \mathcal{B}(\lambda) \) and the reverse tableaux.
As applications of our framework, we present concrete combinatorial realizations of the crystal embedding of $\mathcal B(\lambda)$ into $\mathcal B(\infty)$ and the set of Lusztig data.
\end{abstract}
\maketitle
\tableofcontents
\section*{Introduction}
The crystal basis, introduced by Kashiwara \cite{Kas90,Kas91}, serves as a fundamental combinatorial tool in the study of the representation theory of quantum groups $U_q(\mathfrak g)$ associated with symmetrizable Kac-Moody Lie algebras $\mathfrak g$. The crystal bases can be viewed as bases at $q=0$, endowed with the structure of colored oriented graphs, known as crystal graphs. 
These graphs consist of nodes and arrows, where the nodes correspond to elements of the crystal bases, and the arrows are labeled by the Kashiwara operators.
The combinatorial properties of crystal graphs reflect the intrinsic combinatorial structure of quantum groups and their integrable modules.
Accordingly, providing a combinatorial realization of crystal graphs that is independent of the underlying algebraic structure of crystal bases has become a key research direction \cite{KN20,KN21,KN94,Lee07,Lee14,Na99,NZ97}.

\vskip 2mm

For the classical finite-dimensional simple Lie algebras $\mathfrak g$, Kashiwara and Nakashima provided a combinatorial realization of the crystal graphs for the highest weight crystals $\mathcal B(\lambda)$ using modified semi-standard Young tableaux \cite{KN94}. For type $A_n$, the nodes of the crystal graph of $\mathcal B(\lambda)$ are characterized by semi-standard Young tableaux (SSYT), with the action of Kashiwara operators described using the signature rule and admissible reading of tableaux. In \cite{Lee07}, Lee provided a realization of the crystal basis \( \mathcal{B}(\infty) \) of type \( A_n \) using the  marginally large tableau model. More generally, for the symmetrizable Kac-Moody Lie algebras $\mathfrak g$, Nakashima and Zelevinsky gave a polyhedral realization of the crystal bases $\mathcal B(\infty)$ and $\mathcal B(\lambda)$ using the integer points in certain polytopes \cite{NZ97, Na99}. In the polyhedral realizations, the nodes of the crystal graphs are described by sequences of integers that satisfy certain inequalities.

\vskip 2mm

In \cite{KN20, KN21}, Kanakubo and Nakashima introduced the notion of adapted sequences and employed column tableaux to describe the linear functions arising in the polyhedral realizations, thereby obtaining an explicit characterization of the polyhedral realizations of the crystal bases  $\mathcal B(\infty)$ and $\mathcal B(\lambda)$ in classical finite types. In the combinatorial descriptions of polyhedral realizations, the combinatorial model of column tableaux used by Kanakubo and Nakashima do not correspond directly to crystal base elements. 
Therefore, establishing a direct correspondence between the polyhedral realizations of crystal bases and their associated combinatorial models—without relying on the description of linear functions—represents a natural line of investigation.

\vskip 2mm

The main purpose of this paper is to explore  the correspondence between the polyhedral and Young tableau realizations of the  crystal bases $\mathcal B(\lambda)$ and $\mathcal B(\infty)$ in type $A_n$. By computing the weight functions in the polyhedral realizations, we observe that they does not coincide with the weight functions on semi-standard Young tableaux and marginally large tableaux. To resolve this discrepancy, we introduce the concept of reverse semi-standard Young tableaux (RSSYT) and reverse marginally large tableaux  (RMLT), which form special classes of plane partitions. 

\vskip 2mm

Subsequently, we obtain a crystal structure on the set of RSSYT (resp. RMLT) from the crystal structure on the set of semi-standard Young tableaux (resp. marginally large tableaux). Furthermore, by employing the models of reverse tableaux, we demonstrate that the crystal graph of $\mathcal B(\lambda)$ admits central symmetry, while the crystal graph of $\mathcal B(\infty)$ exhibits left-right (mirror) symmetry. Finally, we construct explicit correspondences between RSSTY (resp. RMLT) and the integer sequences appearing in the polyhedral realizations of $\mathcal B(\lambda)$ (resp. $\mathcal B(\infty)$), and prove that these correspondences are crystal isomorphisms.

\vskip 2mm

It is well known that there is a one-to-one correspondence between SSTY and Gelfand-Tsetlin patterns. Therefore, the crystal structure on the set of SSTY induces a crystal structure on the set of Gelfand–Tsetlin patterns. In \cite{HK20}, Hartwig and Kingston defined a crystal structure on the set of Gelfand–Tsetlin patterns and demonstrated that this structure is compatible with the crystal structure of $\mathcal B(\lambda)$. In this paper, we derive the crystal structure on the set of Gelfand–Tsetlin patterns in a natural way, using the correspondence between RSSTY and the integer sequences in the polyhedral realization of $\mathcal B(\lambda)$.

\vskip 2mm

\vskip 2mm

On the other hand, when \( q \to 0 \), the crystal basis \( \mathcal{B}(\infty) \) corresponds to the PBW basis of $U^-_q(\mathfrak g)$, which can be parametrized by the Lusztig data \( \mathbf B_{\mathbf i} \). Therefore, a combinatorial description of the crystal embedding \( \mathcal{B}(\lambda) \hookrightarrow \mathcal{B}(\infty) \otimes R_\lambda \) (here $R_\lambda$ denotes the one-point crystal; see Example \ref{ex:Rlambda})
is helpful for understanding the combinatorial structure of PBW basis. In this direction, Lee \cite{Lee14} imposed specific constraints on marginally large tableaux and thereby produced a large tableau realization of \( \mathcal{B}(\lambda) \).

\vskip 2mm
Building on explicit correspondences among polyhedral realizations, RSSYT and Gelfand--Tsetlin patterns, we construct RMLT directly from RSSYT. This identifies a natural subset of \( \mathcal T'(\infty) \) (the set of all RMLT) that carries the crystal structure of \( \mathcal B(\lambda) \). Consequently, we obtain a purely combinatorial realization of the embedding $
\mathcal B(\lambda)\hookrightarrow \mathcal B(\infty)\otimes R_\lambda$,
which, from the reverse–tableau viewpoint, recovers Lee’s construction.

\vskip 2mm
In \cite{Kwon18}, Kwon gave a combinatorial description of the embedding
\(\psi_{\lambda}^{\mathbf i}:\mathcal B(\lambda)\otimes T_{-\lambda}\hookrightarrow \mathbf B_{\mathbf i}\)
associated with a Dynkin quiver of type \(A_{n-1}\) having a single sink.
As an application of our combinatorial realization of the crystal embedding,
we provide an alternative, purely combinatorial description of
\(\psi_{\lambda}^{\mathbf i}\) for the standard reduced word
\(\mathbf i_0=(1,2,1,3,2,1,\dots,n,\dots,2,1)\).
The advantage of our description is that the Lusztig data of an RSSYT can be read off directly from the tableau.

\vskip 2mm
The paper is organized as follows. 
In Section~\ref{sec:Ac}, we recall the notion of abstract crystals for quantum groups in type $A_n$. 
In Section~\ref{sec:YTRB}, we revisit the Young–tableau realizations of the crystals $\mathcal B(\infty)$ and $\mathcal B(\lambda)$ in type $A_n$, and introduce the \emph{reverse} marginally large tableaux (RMLT) and \emph{reverse} semi–standard Young tableaux (RSSYT). Transporting the crystal structure from MLT (resp.\ SSYT) yields crystal structures on  RMLT (resp.\ RSSYT).
Section~\ref{sec:tableau_polyhedral_Binfty_Blambda} reviews the polyhedral realizations of $\mathcal B(\infty)$ and $\mathcal B(\lambda)$ in type $A_n$, and constructs explicit correspondences between the reverse--tableau models and the integer sequences appearing in polyhedral realizations. It is shown that these correspondences are crystal isomorphisms. In Section~\ref{sec:CGTP}, a crystal structure on the set of Gelfand--Tsetlin patterns is obtained from the reverse--tableau description of the polyhedral realization of $\mathcal B(\lambda)$.
Finally, Section~\ref{sec:crystal embedding} presents a combinatorial characterization of the embedding of $\mathcal B(\lambda)$ into both $\mathcal B(\infty)$ and the set formed by  Lusztig data.

\vskip 2mm
{\bf Acknowledgements.}
This work was supported by the China Postdoctoral Science Foundation under Grant Number 2024M760061 and NSFC of grant No.12501041.
The authors would like to thank the anonymous referee for the careful reading of the manuscript and for the valuable comments and suggestions, which helped improve the presentation of this paper.
\section{Abstract crystal}\label{sec:Ac}

Let \(I=\{1,2,\dots,n\}\). In the sequel, we consider the Cartan matrix \(A\) of type \(A_n\), i.e., \(A=(a_{ij})_{i,j\in I}\) with \(a_{ij}=2\delta_{ij}-\delta_{i,j+1}-\delta_{i+1,j}\). In this case, the simple roots are given by \(\alpha_i=\epsilon_i-\epsilon_{i+1}\) for \(i\in I\), and the fundamental weights are identified with \(\omega_i=\sum_{j=1}^i \epsilon_j\), where \(\{\epsilon_1,\epsilon_2,\dots,\epsilon_{n+1}\}\) is the standard basis of \(\mathbb{R}^{n+1}\).

\vskip 1mm

The weight lattice is \(P=\bigoplus_{i=1}^n \mathbb{Z}\,\omega_i\), and its dual lattice is \(P^\vee:=\mathrm{Hom}_{\mathbb{Z}}(P,\mathbb{Z})\). The simple coroots \(\{\alpha_i^\vee\}_{i\in I}\subset P^\vee\) satisfy \(\langle \alpha_i^\vee,\alpha_j\rangle=a_{ij}\) for all \(i,j\in I\). We also write \(P^+=\bigoplus_{i=1}^n \mathbb{Z}_{\geq 0}\,\omega_i\) for the set of dominant integral weights.

\vskip 1mm

Let \(q\) be an indeterminate. The quantum group \(U:=U_q(\mathfrak{sl}_{n+1})\) is the unital associative \(\mathbb{Q}(q)\)-algebra generated by \(e_i, f_i\) \((i\in I)\) and \(K_i^{\pm 1}\) \((i\in I)\), subject to the following relations:
\begin{enumerate}
	\item[(1)] \(K_iK_i^{-1}=1\) and \(K_iK_j=K_jK_i\) for all \(i,j\in I\).
	
	\item[(2)] \(K_i e_j K_i^{-1}=q^{a_{ij}}e_j\) and \(K_i f_j K_i^{-1}=q^{-a_{ij}}f_j\) for all \(i,j\in I\).
	
	\item[(3)] \(e_i f_j-f_j e_i=\delta_{ij}\dfrac{K_i-K_i^{-1}}{q-q^{-1}}\) for all \(i,j\in I\).
	
	\item[(4)] \(e_i e_j=e_j e_i\) and \(f_i f_j=f_j f_i\) for all \(i,j\in I\) with \(|i-j|>1\).
	
	\item[(5)] \(e_i^2 e_j+e_j e_i^2=(q+q^{-1})\,e_i e_j e_i\) for all \(i,j\in I\) with \(|i-j|=1\).
	
	\item[(6)] \(f_i^2 f_j+f_j f_i^2=(q+q^{-1})\,f_i f_j f_i\) for all \(i,j\in I\) with \(|i-j|=1\).
\end{enumerate}

There exist the triangular decomposition of $U$ given by $U\cong U^-\otimes U^0\otimes U^+$, where
$U^+$ (resp.\ $U^-$) is the subalgebra of $U$ generated by $\{e_i\mid i\in I\}$ (resp.\ $\{f_i\mid i\in I\}$),
and $U^0$ is the subalgebra generated by $\{K_i^{\pm1}\mid i\in I\}$.

\begin{defn}\cite[Definition 3.1]{JKKS}\label{def:abstract crystal}
	An {\it abstract $U$-crystal}   is
	a set $B$ together with the maps $\wt : B \rightarrow P$, $\tilde{e}_i,
	\tilde{f}_i: B \rightarrow B \cup \{ 0 \}$ and $\varepsilon_i,
	\varphi_i : B \rightarrow \mathbb{Z} \cup \{-\infty\}$  $(i \in I)
	$ satisfying the following conditions:
	\begin{itemize}
		\item[(i)] $\varphi_{i}(b) = \varepsilon_{i}(b) + \langle \alpha_i^\vee, \wt
		(b) \rangle$ \ for all  $ i \in I$,
		
		\item[(ii)] $ \wt(\tilde {e}_i b) = \wt (b) + \alpha_i$ \ if  $\tilde{e}_i b \in B$,
		
		\item[(iii)] $ \wt(\tilde{f}_i b) = \wt(b) - \alpha_i$ \ if  $\tilde{f}_ib \in B$,
		
		\item[(iv)] $\varepsilon_i(\tilde{e}_i b) = \varepsilon_i(b) - 1$, \
		$\varphi_i(\tilde{e}_i b) = \varphi_i(b) + 1$ \ if  $\tilde{e}_ib \in
		B$,

		\item[(v)] $\varepsilon_i(\tilde{f}_i b) = \varepsilon_i(b) + 1$, \
		$\varphi_i(\tilde{f}_i b) = \varphi_i(b) - 1$ \ if  $\tilde{f}_ib \in
		B$,
			
		\item[(vi)] $\tilde{f}_ib=b'$ \ if and only if  \ $b=\tilde{e}_ib'$ for  $b,b'\in B,\ i \in I $,
		
		\item[(vii)] If $\varphi_i(b) = -\infty$  for  $b\in B$, then
		$\tilde{e}_i b = \tilde{f}_i b = 0$.
		
	\end{itemize}
\end{defn}

\begin{Ex}\label{ex:Rlambda}
For $\lambda\in P$, let $R_\lambda:=\{r_\lambda\}$ be the set consisting of a single vector $r_\lambda$ with the following maps:
\begin{equation*}
\wt(r_\lambda)=\lambda,\quad \varepsilon_i(r_\lambda)=-\langle\alpha_i^\vee,\lambda\rangle,\quad \varphi_i(r_\lambda)=0,\quad \tilde{e}_i(r_\lambda)=\tilde{f}_i(r_\lambda)=0.
\end{equation*}
Then $R_\lambda$ is a crystal.
\end{Ex}

\begin{defn}
	Let $B_1$ and $B_2$ be crystals.
	A {\it crystal morphism} (or {\it
		morphism of crystals}) $\Psi : B_1 \rightarrow B_2$ is a map
	$\Psi : B_1 \cup \{0\} \rightarrow B_2 \cup \{0\}$ such that
	\begin{enumerate}
		\item[(i)] $\Psi(0)=0$,
		\item[(ii)] if $b \in B_1$ and $\Psi(b) \in B_2$, then $\wt(\Psi(b))=\wt(b)$, $\varepsilon_i(\Psi(b))=\varepsilon_i(b)$, and $\varphi_i(\Psi(b))=\varphi_i(b)$ for all $i \in I$,
		\item[(iii)] if $b, b' \in B_1$, $\Psi(b), \Psi(b') \in B_2$ and $\tilde{f}_i b=b'$, then $\tilde{f}_i \Psi(b)=\Psi(b')$ and $\Psi(b)=\tilde{e}_i \Psi(b')$ for all $i \in I$.
	\end{enumerate}
	A crystal morphism $\Psi : B_1 \rightarrow B_2$ is called an {\it isomorphism} if it is a bijection from $B_1 \cup \{0\}$ to $B_2 \cup \{0\}$.
\end{defn}

For two crystals  $B_1$ and $B_2 $, we define their {\it tensor product} $B_1 \otimes B_2$  as the set 
$B_1 \times B_2 $, with the crystal structure defined as follows:
\begin{equation*}
	\begin{aligned}
		\tilde {e}_i(b_1 \otimes b_2)&= \begin{cases} \tilde {e}_i b_1 \otimes b_2 &\text{if \;$\varphi_i(b_1)\ge \varepsilon_i(b_2)$}, \\
			b_1\otimes \tilde {e}_i b_2 & \text{if\; $\varphi_i(b_1) < \varepsilon_i(b_2)$}, \end{cases} \\
		\tilde {f}_i(b_1 \otimes b_2)&= \begin{cases} \tilde {f}_i b_1 \otimes b_2
			&\text{if \; $\varphi_i(b_1) > \varepsilon_i(b_2)$}, \\
			b_1\otimes \tilde {f_i} b_2 & \text{if \; $\varphi_i(b_1) \le
				\varepsilon_i(b_2)$},
		\end{cases}\\
		\wt(b_1 \otimes b_2)&= \wt(b_1)+\wt(b_2),\\
		\varepsilon_{i}(b_1 \otimes b_2)&= \max(	\varepsilon_{i}(b_1), 	\varepsilon_i(b_2) -
		\langle \alpha_i^\vee, \wt(b_1) \rangle ),\\
		\varphi_{i}(b_1 \otimes b_2)& = \max(\varphi_{i}(b_2), \varphi_i(b_1)
		+\langle \alpha_i^\vee, \wt(b_2)\rangle ).
	\end{aligned}
\end{equation*}

Under the tensor product rule, an abstract crystal structure is defined on $B_1\otimes B_2$ (\cite[Lemma 3.10]{JKKS}). More generally, an abstract crystal structure exists on $B_1\otimes B_2\otimes\cdots \otimes B_N$ for $N\geq 1$ (\cite[Lemma 3.11]{JKKS}, \cite[Proposition 2.1.1]{KN94}).

\vskip 10mm

\section{Young tableau realizations of  $\mathcal B(\infty)$ and $\mathcal B(\lambda)$ }\label{sec:YTRB}

\subsection{Marginally large tableau model of $\mathcal B(\infty)$}

In this section, we review the notation of marginally large tableaux introduced in \cite{Lee07}.

Recall that a semistandard Young tableau (SSYT) is a Young diagram filled with entries such that the rows are weakly increasing and the columns are strictly increasing. We regard each box together with its entry as a colored box; in particular, a box of color $i$
is called an $i$-box.

\begin{defn}
	A semi-standard Young tableau $T$ with $n$ non-empty rows and entries in $I\cup\{n+1\}$ is called {\it marginally large}, if for any $i\in I$, the number of $i$-boxes in the $i$-th row of $T$ is greater than the number of all boxes in the $(i+1)$-th row by exactly one. 
\end{defn}

For a given marginally large semi-standard Young tableau, we add infinite number of leftmost $i$-box to the leftmost side of each row, thus obtaining an infinite tableau, which is called {\it marginally large tableau} (MLT). 

\vskip 2mm

Let $y_j^i(T)$ $(i>j)$ denote the number of $i$-boxes in $j$-th row of a MLT $T$. If there is no risk of confusion, we shall denote $y_j^i(T)$ by $y_j^i$ for brevity.
\begin{Ex}\label{ex:LTA4}
	For the Cartan datum of type $A_4$, the infinite tableau in Figure \ref{fig:MLTA4} is a marginally large tableau.	
	\begin{figure}[H] 
		\begin{tikzpicture}[scale=0.34]
		\draw (0,0)--(46,0);
		\draw (0,-1)--(46,-1);
		\draw (0,-2)--(29,-2);	
		\draw (0,-3)--(16,-3);	
		\draw (0,-4)--(7,-4);	
		\draw (1,0)--(1,-4);	
		\draw (2,0)--(2,-4);		
		\draw (3,0)--(3,-4);			
		\draw (4,0)--(4,-4);	
		\draw (6,0)--(6,-4);			
		\draw (7,0)--(7,-4);			
		\draw (8,0)--(8,-3);		
		\draw (9,0)--(9,-3);		
		\draw (11,0)--(11,-3);	
		\draw (12,0)--(12,-3);		
		\draw (13,0)--(13,-3);	
		\draw (15,0)--(15,-3);			
		\draw (16,0)--(16,-3);			
		\draw (17,0)--(17,-2);	
		\draw (18,0)--(18,-2);			
		\draw (20,0)--(20,-2);			
		\draw (21,0)--(21,-2);			
		\draw (22,0)--(22,-2);				
		\draw (24,0)--(24,-2);			
		\draw (25,0)--(25,-2);			
		\draw (26,0)--(26,-2);			
		\draw (28,0)--(28,-2);			
		\draw (29,0)--(29,-2);
		\draw (30,0)--(30,-1);
		\draw (31,0)--(31,-1);
		\draw (33,0)--(33,-1);
		\draw (34,0)--(34,-1);
		\draw (35,0)--(35,-1);
		\draw (37,0)--(37,-1);		
		\draw (38,0)--(38,-1);			
		\draw (39,0)--(39,-1);			
		\draw (41,0)--(41,-1);		
		\draw (42,0)--(42,-1);			
		\draw (43,0)--(43,-1);
		\draw (45,0)--(45,-1);		
		\draw (46,0)--(46,-1);			
		
		\node at (0.5,-0.5){\small ...};
		\node at (0.5,-1.5){\small ...};		
		\node at (0.5,-2.5){\small ...};	
		\node at (0.5,-3.5){\small ...};	
		
		\node at (1.5,-0.5){\small 1};
		\node at (1.5,-1.5){\small 2};		
		\node at (1.5,-2.5){\small 3};	
		\node at (1.5,-3.5){\small 4};	
		
		\node at (2.5,-0.5){\small 1};
		\node at (2.5,-1.5){\small 2};		
		\node at (2.5,-2.5){\small 3};	
		\node at (2.5,-3.5){\small 4};	
		
		\node at (3.5,-0.5){\small 1};
		\node at (3.5,-1.5){\small 2};		
		\node at (3.5,-2.5){\small 3};	
		\node at (3.5,-3.5){\small 5};
		
		\node at (5,-0.5){\small ...};
		\node at (5,-1.5){\small ...};		
		\node at (5,-2.5){\small ...};	
		\node at (5,-3.5){\small ...};
		
		\node at (6.5,-0.5){\small 1};
		\node at (6.5,-1.5){\small 2};		
		\node at (6.5,-2.5){\small 3};	
		\node at (6.5,-3.5){\small 5};
		
		\node at (7.5,-0.5){\small 1};
		\node at (7.5,-1.5){\small 2};		
		\node at (7.5,-2.5){\small 3};
		
		\node at (8.5,-0.5){\small 1};
		\node at (8.5,-1.5){\small 2};		
		\node at (8.5,-2.5){\small 4};
		
		\node at (10,-0.5){\small ...};
		\node at (10,-1.5){\small ...};		
		\node at (10,-2.5){\small ...};	
		
		\node at (11.5,-0.5){\small 1};
		\node at (11.5,-1.5){\small 2};		
		\node at (11.5,-2.5){\small 4};
		
		\node at (12.5,-0.5){\small 1};
		\node at (12.5,-1.5){\small 2};		
		\node at (12.5,-2.5){\small 5};
		
		\node at (14,-0.5){\small ...};
		\node at (14,-1.5){\small ...};		
		\node at (14,-2.5){\small ...};
		
		\node at (15.5,-0.5){\small 1};
		\node at (15.5,-1.5){\small 2};		
		\node at (15.5,-2.5){\small 5};
		
		\node at (16.5,-0.5){\small 1};
		\node at (16.5,-1.5){\small 2};	
		
		\node at (17.5,-0.5){\small 1};
		\node at (17.5,-1.5){\small 3};	
		
		\node at (19,-0.5){\small ...};
		\node at (19,-1.5){\small ...};	
		
		\node at (20.5,-0.5){\small 1};
		\node at (20.5,-1.5){\small 3};	
		
		\node at (21.5,-0.5){\small 1};
		\node at (21.5,-1.5){\small 4};	
		
		\node at (23,-0.5){\small ...};
		\node at (23,-1.5){\small ...};	
		
		\node at (24.5,-0.5){\small 1};
		\node at (24.5,-1.5){\small 4};	
		
		\node at (25.5,-0.5){\small 1};
		\node at (25.5,-1.5){\small 5};	
		
		\node at (27,-0.5){\small ...};
		\node at (27,-1.5){\small ...};	
		
		\node at (28.5,-0.5){\small 1};
		\node at (28.5,-1.5){\small 5};	
		
		\node at (29.5,-0.5){\small 1};
		\node at (30.5,-0.5){\small 2};
		\node at (32,-0.5){\small ...};
		\node at (33.5,-0.5){\small 2};
		\node at (34.5,-0.5){\small 3};
		\node at (36,-0.5){\small ...};
		\node at (37.5,-0.5){\small 3};
		\node at (38.5,-0.5){\small 4};
		\node at (40,-0.5){\small ...};
		\node at (41.5,-0.5){\small 4};
		\node at (42.5,-0.5){\small 5};
		\node at (44,-0.5){\small ...};
		\node at (45.5,-0.5){\small 5};
		
		\draw[decorate,decoration={mirror,brace}] (3,-4.2)--(7,-4.2);
		\node at (5,-5.3){\small $y_4^5$};
		
		\draw[decorate,decoration={mirror,brace}] (8,-3.2)--(12,-3.2);
		\node at (10,-4.3){\small $y_3^4$};
		
		\draw[decorate,decoration={mirror,brace}] (12,-3.2)--(16,-3.2);
		\node at (14,-4.3){\small $y_3^5$};
		
		\draw[decorate,decoration={mirror,brace}] (17,-2.2)--(21,-2.2);
		\node at (19,-3.3){\small $y_2^3$};
		
		\draw[decorate,decoration={mirror,brace}] (21,-2.2)--(25,-2.2);
		\node at (23,-3.3){\small $y_2^4$};
		
		\draw[decorate,decoration={mirror,brace}] (25,-2.2)--(29,-2.2);
		\node at (27,-3.3){\small $y_2^5$};
		
		\draw[decorate,decoration={mirror,brace}] (30,-1.2)--(34,-1.2);
		\node at (32,-2.3){\small $y_1^2$};
		
		\draw[decorate,decoration={mirror,brace}] (34,-1.2)--(38,-1.2);
		\node at (36,-2.3){\small $y_1^3$};
		
		\draw[decorate,decoration={mirror,brace}] (38,-1.2)--(42,-1.2);
		\node at (40,-2.3){\small $y_1^4$};
		
		\draw[decorate,decoration={mirror,brace}] (42,-1.2)--(46,-1.2);
		\node at (44,-2.3){\small $y_1^5$};		
		\end{tikzpicture}
		\vskip -3.5mm
		\caption{Marginally large tableau of type $A_4$}\label{fig:MLTA4}
	\end{figure}
\end{Ex}

\begin{Rmk}
	A marginally large tableau $Y$ can be uniquely determined by the values of the sequence  $(y_j^i)_{i>j}$ for $j\in I$ and $i\in \{2,3,\cdots,n+1\}$.
\end{Rmk}

Let $\mathcal T(\infty)$ denote the set consisting of all marginally large tableaux $T$. The Kashiwara operators $\tilde{f}_i$ and $\tilde{e}_i$ on $\mathcal T(\infty)$ are defined as follows (\cite[Section 4]{Lee07}):
\begin{enumerate}
	\item[(i)]  
	Consider the infinite sequence of colored boxes obtained by applying the Far-Eastern reading to $T\in \mathcal T(\infty)$. For each entry \( b \) in this sequence and for a fixed index \( i \in I \), we assign a sign \textbf{$``-"$} to the box if \( b = i+1 \), and a sign \textbf{$``+"$} if \( b = i \);  
	otherwise, we assign nothing. 
From this sequence of $+$'s and $-$'s, we cancel out all $(+,-)$ pairs. The remaining sequence $-$'s followed by $+$'s is called the \textit{$i$-signature} of $T$.
	
	\vskip 1mm
	
	\item[(ii)] Denote by $T'$ the tableau obtained from $T$ by replacing the entry $i$ by $i+1$  in the box corresponding to the leftmost $+$ in the $i$-signature of $T$.
	\begin{enumerate}
		\item[(a)]  If $T'$ is marginally large, then we define $\tilde{f}_i{T}$ to be $T'$.
		\item[(b)] If $T'$ is not marginally large, then we define $\tilde{f}_i{T}$ to be the MLT obtained by pushing all the rows appearing below the changed box in $T'$ to the left by one box.
	\end{enumerate}
	
	\vskip 1mm
	
	\item[(iii)] Denote by $T''$ the tableau obtained from $T$ by replacing the entry $i$ by $i-1$  in the box corresponding to the rightmost $-$ in the $i$-signature of $T$.
	\begin{enumerate}
		\item[(a)]  If $T''$ is marginally large, then we define $\tilde{e}_i{T}$ to be $T''$.
		\item[(b)] If $T''$ is not marginally large, then we define $\tilde{e}_i{T}$ to be the MLT obtained by pushing all the rows appearing below the changed box in $T''$ to the right by one box.
	\end{enumerate}
	
	\vskip 1mm
	
	\item[(iv)] If there is no $-$ in the $i$-signature of $T$, then we define $
	\tilde{e}_iT=0$.
	
\end{enumerate}

\vskip 1mm

We define $\wt:\mathcal T(\infty)\to P$, $\varepsilon_i, \varphi_i:\mathcal T(\infty)\to\mathbb Z$ as follows:
\begin{equation}\label{eq:wt_epsilon_phi}
\begin{aligned}
&\wt(T)=-\sum_{i=1}^n(\sum_{k=1}^i(\sum_{l=i+1}^{n+1}y_k^l))\alpha_i,\\
&\varepsilon_i(T)=\text{the number of $-$'s in the $i$-signature of $T$},\\
&\varphi_i(T)=\varepsilon_i(T)+\langle\alpha_i^\vee,\wt(T)\rangle.
\end{aligned}
\end{equation}

\vskip 1mm

Based on the structure of the marginally large tableau \( T \), we can conclude that
\begin{equation}\label{eq:epsilonTinfty}
	\varepsilon_i(T)=\max_{1\leq j\leq i}\{\sum_{k=1}^j(y_k^{i+1}-y_{k-1}^{i})\}.
\end{equation}

\vskip 1mm

\begin{Ex}
Fix $n=4$ and consider the following marginally large tableau $T$:
\[
\raisebox{-1.5\baselineskip}{$T=\ $}\; \scalebox{0.9}
{$ 
\begin{ytableau} 
	1&1&1&1&1&1&1&1&1&2\\ 
	2&2&2&2&2&3&3&4\\ 
	3&3&3&4\\ 
	4&5 
\end{ytableau} 
$}
\]

Applying the Far--Eastern reading to $T$, a word of $T$ together with the signs for $2,3$ and the $2$-signature of $T$ is obtained as follows:

\[
\begin{array}{r@{}*{24}{@{\hspace{4pt}}c}} w(T)= &2&1&1&4&1&3&1&3&1&2&1&2&4&1&2&3&1&2&3&5&1&2&3&4\\[0.6ex] &+& & & & &-& &-& &+& &+& & &+&-& &+&-& & &+&-&\\[1.2ex] & & & & & & & &-& &+& &+& & & & & & & & & & & & \end{array}
\]

Then the actions of $\tilde f_2$ and $\tilde e_2$ on $T$ are given by

\[
\raisebox{-1.5\baselineskip}{$\tilde f_2T=\ $} \scalebox{0.9}{$ \begin{ytableau} 1&1&1&1&1&1&1&1&1&2\\ 2&2&2&2&3&3&3&4\\ 3&3&4\\ 5 \end{ytableau} $} \quad\qquad \raisebox{-1.5\baselineskip}{$\tilde e_2T=\ $} \scalebox{0.9}{$ \begin{ytableau} 1&1&1&1&1&1&1&1&1&2\\ 2&2&2&2&2&2&3&4\\ 3&3&3&3&4\\ 4&4&5 \end{ytableau} $}
\]

\vskip 2mm

It follows from \eqref{eq:wt_epsilon_phi} that
\begin{align*}
	&\wt(T) = -(\alpha_1+3\alpha_2+2\alpha_3+\alpha_4),\\
	&(\varepsilon_1(T),\varepsilon_2(T),\varepsilon_3(T),\varepsilon_4(T)) 
	= (1,1,1,0),\\
	&(\varphi_1(T),\varphi_2(T),\varphi_3(T),\varphi_4(T)) 
	= (2,-2,1,0).
\end{align*}

\end{Ex}

\vskip 4mm

Let $\mathcal B(\infty)$ be the crystal of $U^-$ (\cite[Section 3.5]{Kas91}), then we have the following theorem.
\begin{Thm}\cite[Theorem 4.8]{Lee07}\label{thm:realizationBinfty}
	The set $\mathcal T(\infty)$ with the maps $\tilde{e}_i$, $\tilde{f}_i$, $\varepsilon_i$, $\varphi_i$ $(i\in I)$ and  $\wt$ form a crystal of type $A_n$. There exists an $U_q(A_n)$-crystal isomorphism:
	\begin{align*}
	\mathcal B(\infty)\to \mathcal T(\infty),\quad u_{\infty}\to T_{\infty},
	\end{align*}
	where $u_\infty\in\mathcal B(\infty)$ is the vector corresponding to $1\in U^-_q(A_n)$ and $T_\infty$ is the marginally large tableau such that $y_j^i=0$ for $i>j$, $j\in I$ and $i\in \{2,3,\cdots,n+1\}$.
\end{Thm}

\subsection{Young tableau model of $\mathcal B(\lambda)$}
Let \(\lambda \in P^+\) be a dominant integral weight. Then \(\lambda\) can be expressed as
$
\lambda = \sum_{i=1}^n a_i \omega_i$ with  $a_i \geq 0$.
Define
$
\lambda_i := \sum_{j=i}^n a_j$ for $1 \leq i \leq n$.
With this, \(\lambda\) can also be written in terms of the standard basis \(\{\epsilon_i\}\) as
$
\lambda = \sum_{i=1}^n \lambda_i \epsilon_i$.
Therefore, the dominant integral weight $\lambda$ corresponds to a Young diagram $(\lambda_1,\lambda_2,\cdots,\lambda_n,0,0,\cdots)$. 

\vskip 1mm

By a slight abuse of notation, we shall also denote by \(\lambda\) the Young diagram corresponding to the dominant weight \(\lambda\).  
Let \(\mathcal{T}(\lambda)\) denote the set of all semi-standard Young tableaux (SSYT) of shape \(\lambda\) with entries in the set \(\{1, 2, \ldots, n+1\}\).

\vskip 1mm

The Kashiwara operators \(\tilde{e}_i\) and \(\tilde{f}_i: \mathcal{T}(\lambda) \to \mathcal{T}(\lambda) \cup \{0\}\) are defined in the same way as in the case of \(\mathcal{T}(\infty)\), except that the condition of marginally large is not required.

\vskip 1mm

For any $ T\in \mathcal T(\lambda)$, we define $\wt : \mathcal T(\lambda) \rightarrow P$ and $\varepsilon_i,
\varphi_i : \mathcal T(\lambda) \rightarrow \mathbb{Z} \cup \{-\infty\}$  $(i \in I)
$ as follows:
\begin{equation}\label{eq:crystal on Ylambda}
\begin{aligned}
\wt(T)&=\sum_{1\leq i\leq n+1}|T|_i\epsilon_i,\ \text{where}\ |Y|_i\ \text{is the number of $i$-boxes in}\ T,  \\
\varepsilon_i(T)&=\text{the number of $``-"$ in the $i$-signature of $T$},\\
\varphi_i(T)&=\text{the number of $``+"$ in the $i$-signature of $T$}.
\end{aligned}
\end{equation}

Let $\mathcal B(\lambda)$ be the crystal of the highest weight module $V(\lambda)$ over quantum group $U_q(A_n)$ (\cite[Section 4]{Kas90}).

\begin{Thm}\cite[Theorem 3.4.2]{KN94}
For any \(\lambda \in P^+\), the set \(\mathcal{T}(\lambda)\), together with the maps \(\tilde{e}_i\), \(\tilde{f}_i\), \(\varepsilon_i\), \(\varphi_i\) \((i \in I)\), and \(\mathrm{wt}\), forms a crystal of type \(A_n\). Moreover, the highest weight crystal \(\mathcal{B}(\lambda)\) is isomorphic to the crystal \(\mathcal{T}(\lambda)\).
\end{Thm}

\vskip 2mm

\subsection{Reverse tableau model for $\mathcal B(\infty)$}

For each marginally large tableau \( T \) of type \( A_n \), we define the corresponding reverse marginally large tableau (RMLT) \( T' \) by replacing each entry \( a \) in \( T \) with \( n+2 - a \). Let \( z_i^j := y_i^{n+2-j} \) for \( 1 \leq i \leq n \) and \( 2 \leq i + j \leq n + 1 \).

\vskip 1mm

For instance, the RMLT corresponding to the marginally large tableau given in  Example \ref{ex:LTA4} is presented in Figure \ref{fig:RMLTA4}.
\begin{figure}[H] 
	\begin{tikzpicture}[scale=0.34]
		\draw (0,0)--(46,0);
		\draw (0,-1)--(46,-1);
		\draw (0,-2)--(29,-2);	
		\draw (0,-3)--(16,-3);	
		\draw (0,-4)--(7,-4);	
		\draw (1,0)--(1,-4);	
		\draw (2,0)--(2,-4);		
		\draw (3,0)--(3,-4);			
		\draw (4,0)--(4,-4);	
		\draw (6,0)--(6,-4);			
		\draw (7,0)--(7,-4);			
		\draw (8,0)--(8,-3);		
		\draw (9,0)--(9,-3);		
		\draw (11,0)--(11,-3);	
		\draw (12,0)--(12,-3);		
		\draw (13,0)--(13,-3);	
		\draw (15,0)--(15,-3);			
		\draw (16,0)--(16,-3);			
		\draw (17,0)--(17,-2);	
		\draw (18,0)--(18,-2);			
		\draw (20,0)--(20,-2);			
		\draw (21,0)--(21,-2);			
		\draw (22,0)--(22,-2);				
		\draw (24,0)--(24,-2);			
		\draw (25,0)--(25,-2);			
		\draw (26,0)--(26,-2);			
		\draw (28,0)--(28,-2);			
		\draw (29,0)--(29,-2);
		\draw (30,0)--(30,-1);
		\draw (31,0)--(31,-1);
		\draw (33,0)--(33,-1);
		\draw (34,0)--(34,-1);
		\draw (35,0)--(35,-1);
		\draw (37,0)--(37,-1);		
		\draw (38,0)--(38,-1);			
		\draw (39,0)--(39,-1);			
		\draw (41,0)--(41,-1);		
		\draw (42,0)--(42,-1);			
		\draw (43,0)--(43,-1);
		\draw (45,0)--(45,-1);		
		\draw (46,0)--(46,-1);			
		
		\node at (0.5,-0.5){\small ...};
		\node at (0.5,-1.5){\small ...};		
		\node at (0.5,-2.5){\small ...};	
		\node at (0.5,-3.5){\small ...};	
		
		\node at (1.5,-0.5){\small 5};
		\node at (1.5,-1.5){\small 4};		
		\node at (1.5,-2.5){\small 3};	
		\node at (1.5,-3.5){\small 2};	
		
		\node at (2.5,-0.5){\small 5};
		\node at (2.5,-1.5){\small 4};		
		\node at (2.5,-2.5){\small 3};	
		\node at (2.5,-3.5){\small 2};	
		
		\node at (3.5,-0.5){\small 5};
		\node at (3.5,-1.5){\small 4};		
		\node at (3.5,-2.5){\small 3};	
		\node at (3.5,-3.5){\small 1};
		
		\node at (5,-0.5){\small ...};
		\node at (5,-1.5){\small ...};		
		\node at (5,-2.5){\small ...};	
		\node at (5,-3.5){\small ...};
		
		\node at (6.5,-0.5){\small 5};
		\node at (6.5,-1.5){\small 4};		
		\node at (6.5,-2.5){\small 3};	
		\node at (6.5,-3.5){\small 1};
		
		\node at (7.5,-0.5){\small 1};
		\node at (7.5,-1.5){\small 2};		
		\node at (7.5,-2.5){\small 3};
		
		\node at (8.5,-0.5){\small 5};
		\node at (8.5,-1.5){\small 4};		
		\node at (8.5,-2.5){\small 2};
		
		\node at (10,-0.5){\small ...};
		\node at (10,-1.5){\small ...};		
		\node at (10,-2.5){\small ...};	
		
		\node at (11.5,-0.5){\small 5};
		\node at (11.5,-1.5){\small 4};		
		\node at (11.5,-2.5){\small 2};
		
		\node at (12.5,-0.5){\small 5};
		\node at (12.5,-1.5){\small 4};		
		\node at (12.5,-2.5){\small 1};
		
		\node at (14,-0.5){\small ...};
		\node at (14,-1.5){\small ...};		
		\node at (14,-2.5){\small ...};
		
		\node at (15.5,-0.5){\small 5};
		\node at (15.5,-1.5){\small 4};		
		\node at (15.5,-2.5){\small 1};
		
		\node at (16.5,-0.5){\small 5};
		\node at (16.5,-1.5){\small 4};	
		
		\node at (17.5,-0.5){\small 5};
		\node at (17.5,-1.5){\small 3};	
		
		\node at (19,-0.5){\small ...};
		\node at (19,-1.5){\small ...};	
		
		\node at (20.5,-0.5){\small 5};
		\node at (20.5,-1.5){\small 3};	
		
		\node at (21.5,-0.5){\small 5};
		\node at (21.5,-1.5){\small 2};	
		
		\node at (23,-0.5){\small ...};
		\node at (23,-1.5){\small ...};	
		
		\node at (24.5,-0.5){\small 5};
		\node at (24.5,-1.5){\small 2};	
		
		\node at (25.5,-0.5){\small 5};
		\node at (25.5,-1.5){\small 1};	
		
		\node at (27,-0.5){\small ...};
		\node at (27,-1.5){\small ...};	
		
		\node at (28.5,-0.5){\small 5};
		\node at (28.5,-1.5){\small 1};	
		
		\node at (29.5,-0.5){\small 5};
		\node at (30.5,-0.5){\small 4};
		\node at (32,-0.5){\small ...};
		\node at (33.5,-0.5){\small 4};
		\node at (34.5,-0.5){\small 3};
		\node at (36,-0.5){\small ...};
		\node at (37.5,-0.5){\small 3};
		\node at (38.5,-0.5){\small 2};
		\node at (40,-0.5){\small ...};
		\node at (41.5,-0.5){\small 2};
		\node at (42.5,-0.5){\small 1};
		\node at (44,-0.5){\small ...};
		\node at (45.5,-0.5){\small 1};
		
		\draw[decorate,decoration={mirror,brace}] (3,-4.2)--(7,-4.2);
		\node at (5,-5.3){\small $z_4^1$};
		
		\draw[decorate,decoration={mirror,brace}] (8,-3.2)--(12,-3.2);
		\node at (10,-4.3){\small $z_3^2$};
		
		\draw[decorate,decoration={mirror,brace}] (12,-3.2)--(16,-3.2);
		\node at (14,-4.3){\small $z_3^1$};
		
		\draw[decorate,decoration={mirror,brace}] (17,-2.2)--(21,-2.2);
		\node at (19,-3.3){\small $z_2^3$};
		
		\draw[decorate,decoration={mirror,brace}] (21,-2.2)--(25,-2.2);
		\node at (23,-3.3){\small $z_2^2$};
		
		\draw[decorate,decoration={mirror,brace}] (25,-2.2)--(29,-2.2);
		\node at (27,-3.3){\small $z_2^1$};
		
		\draw[decorate,decoration={mirror,brace}] (30,-1.2)--(34,-1.2);
		\node at (32,-2.3){\small $z_1^4$};
		
		\draw[decorate,decoration={mirror,brace}] (34,-1.2)--(38,-1.2);
		\node at (36,-2.3){\small $z_1^3$};
		
		\draw[decorate,decoration={mirror,brace}] (38,-1.2)--(42,-1.2);
		\node at (40,-2.3){\small $z_1^2$};
		
		\draw[decorate,decoration={mirror,brace}] (42,-1.2)--(46,-1.2);
		\node at (44,-2.3){\small $z_1^1$};	
	\end{tikzpicture}
	\vskip -3mm
	\caption{ reverse marginally large tableau of type $A_4$}\label{fig:RMLTA4}
\end{figure}

Let \( \mathcal T'(\infty) \) denote the set of RMLT of type \( A_n \). We define a bijection \( \eta: \mathcal T'(\infty) \to \mathcal T(\infty) \) by subtracting each number in RMLT from \( n+2 \). 

\vskip 1mm

For any $ T\in \mathcal T'(\infty)$, we define $\wt : \mathcal T'(\infty) \rightarrow P$, $\tilde{e}_i,
\tilde{f}_i:\mathcal T'(\infty) \rightarrow \mathcal T'(\infty) \cup \{ 0 \}$ 
and $\varepsilon_i,
\varphi_i : \mathcal T'(\infty) \rightarrow \mathbb{Z} \cup \{-\infty\}$  $(i \in I)
$ as follows:
\begin{equation}\label{eq:crystal on Yrinfty}
	\begin{aligned}
		&\wt(T)=-\sum_{i=1}^n(\sum_{j=1}^{n+1-i}\sum_{k=1}^iz_j^k)\alpha_i,  \\
		&\tilde{e}_iT=\eta^{-1}(\tilde{e}_{n+1-i}\eta(T)),\quad\tilde{f}_iT=\eta^{-1}(\tilde{f}_{n+1-i}\eta(T)),\\
		&\varepsilon_i(T)=\varepsilon_{n+1-i}(\eta(T)),\quad
		\varphi_i(T)=\varphi_{n+1-i}(\eta(T)).
	\end{aligned}
\end{equation}

Then we have the following proposition:
\begin{Prop}\label{prop:infcrystal}
	The crystal structure on $\mathcal T(\infty)$ induces a crystal structure on $\mathcal T'(\infty)$, and the crystal graphs of $\mathcal T(\infty)$ and $\mathcal T'(\infty)$ coincide.
\end{Prop}
\begin{proof}
	Let \( \tau \) be an involution on the weight lattice \( P \) defined by $\tau(\alpha_i)=\alpha_{n+1-i}$ for all $i\in I$. Then for any \( T \in \mathcal T'(\infty) \), we have
	\begin{equation}\label{eq:wt=tauT}
		\begin{aligned}
			\wt(T)&=-\sum_{i=1}^n(\sum_{j=1}^{n+1-i}\sum_{k=1}^iz_j^k)\alpha_i=-\sum_{i=1}^n(\sum_{j=1}^{n+1-i}\sum_{k=1}^iy_j^{n+2-k})\alpha_i\\
			&=-\sum_{i=1}^n(\sum_{j=1}^{n+1-i}\sum_{l=n+2-i}^{n+1}y_j^{l})\alpha_i=-\sum_{k=1}^n(\sum_{j=1}^{k}\sum_{l=k+1}^{n+1}y_j^{l})\alpha_{n+1-k}=\tau(\wt(\eta T)).
		\end{aligned}	
	\end{equation}
	
	From  \( \eqref{eq:wt=tauT} \), it follows that
	\begin{equation*}
		\begin{aligned}
			\wt(\widetilde{f}_iT)=&\wt(\eta^{-1}(\tilde{f}_{n+1-i}\eta(T)))
			=\tau(\wt(\tilde{f}_{n+1-i}\eta(T)))\\
			=&\tau(\wt(\eta T)-\alpha_{n+1-i})=\tau(\tau(\wt(T))-\alpha_{n+1-i})=\wt(T)-\alpha_i.
		\end{aligned}	
	\end{equation*}	
	
	Thus, condition (iii) of Definition \ref{def:abstract crystal} is satisfied.
	The remaining conditions (i), (ii), and (iv)--(vi) can be verified in a similar manner. Hence, the theorem is established.
\end{proof}

Define $\rho_\infty:\Tinf\to\Tinf$ as follows: for $X\in\Tinf$, choose any word
$X=\f_{i_k}\cdots \f_{i_1}\Tinfty$ and set
\begin{equation}\label{eq:def-rho-infty}
	\rho_\infty(X)\ :=\ \eta\!\Bigl(\,\fp_{i_k}\cdots \fp_{i_1}\,\eta^{-1}(\Tinfty)\Bigr).
\end{equation}
Then $\rho_\infty$ is well-defined (independent of the chosen word), and satisfies
\begin{equation}\label{eq:rho-intertwine}
	\rho_\infty\circ \f_i=\f_{\,n+1-i}\circ \rho_\infty\qquad(\forall i\in I),\qquad
	\rho_\infty(\Tinfty)=\Tinfty.
\end{equation}

In analogy with Proposition \ref{prop:rholambda}, we have the following statement.

\begin{Prop}\label{prop:rhoinfty-via-eta}
	For any $i_1,\dots,i_k\in I$,
	\begin{equation}\label{eq:rho-word}
		\rho_\infty\!\bigl(\f_{i_k}\cdots \f_{i_1}\Tinfty\bigr)
		=\f_{\,n+1-i_k}\cdots \f_{\,n+1-i_1}\Tinfty,
	\end{equation}
	and hence $\rho_\infty$ is an involution.
\end{Prop}

\begin{proof}
	Iterating \eqref{eq:rho-intertwine} and using $\rho_\infty(\Tinfty)=\Tinfty$ yields
	\[
	\rho_\infty\!\bigl(\f_{i_k}\cdots\f_{i_1}\Tinfty\bigr)
	=\f_{\,n+1-i_k}\cdots \f_{\,n+1-i_1}\,\rho_\infty(\Tinfty)
	=\f_{\,n+1-i_k}\cdots \f_{\,n+1-i_1}\Tinfty,
	\]
	which is \eqref{eq:rho-word}.
\end{proof}

\begin{Rmk}
	By \eqref{eq:wt=tauT} one also has $\wt(\rho_\infty(X))=\tau(\wt(X))$ for all $X\in\Tinf$.
	This is compatible with the arrow-intertwining \eqref{eq:rho-intertwine}.
\end{Rmk}

Proposition \ref{prop:rhoinfty-via-eta} demonstrates that the crystal graph of $\mathcal B(\infty)$ exhibits a left-right symmetry; specifically, it remains invariant under reflection across the vertical axis passing through the highest weight vector.

\vskip 1mm

For example, the crystal graph of \(\mathcal B(\infty) \) in type \( A_2 \), shown in Figure \ref{BinftyA2}, remains invariant when reflected across the vertical line passing through the top node, with each arrow label replaced by \( 3 \) minus its original value.
\begin{figure}[H]
	\centering
	\includegraphics[width=1\textwidth]{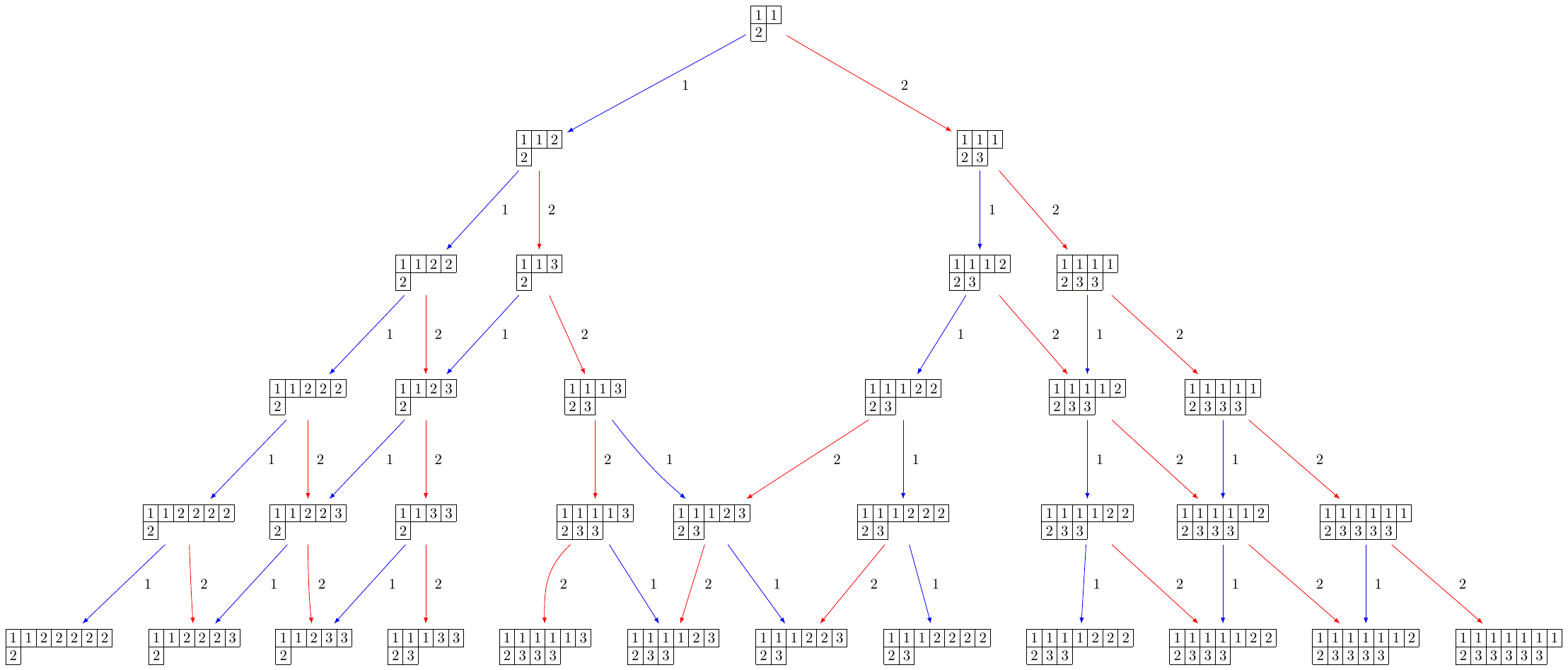}
	\vskip 0.5mm
	\caption{Crystal graph of $\mathcal B(\infty)$ with depth $5$}\label{BinftyA2}
\end{figure}

\subsection{Reverse tableau model for $\mathcal B(\lambda)$}

Let \(\lambda\) be the Young diagram corresponding to a dominant weight \(\lambda \in P^+\).  
We fill the boxes of \(\lambda\) with entries from the set \(\{1, 2, \ldots, n, n+1\}\) such that the entries are weakly decreasing along each row and strictly decreasing down each column.  
A tableau satisfying these conditions is called a \emph{reverse semi-standard Young tableau}.  
Let \(\mathcal{T}'(\lambda)\) denote the set of all reverse semi-standard Young tableaux of shape \(\lambda\).

\vskip 1mm

For any \( T \in \mathcal{T}'(\lambda) \), we define a map $\phi: \mathcal{T}'(\lambda) \to \mathcal{T}(\lambda)$  
by subtracting each entry in \( T \) from \( n+2 \).  
Similarly, we define the inverse map  $\phi^{-1}: \mathcal{T}(\lambda) \to \mathcal{T}'(\lambda)$ 
by subtracting each entry in \( T \in \mathcal{T}(\lambda) \) from \( n+2 \).  
It is straightforward to verify that \(\phi \circ \phi^{-1} = \phi^{-1} \circ \phi = \mathrm{id}\).

\vskip 1mm

For any $T\in \mathcal T'(\lambda)$, we define $\wt: \mathcal T'(\lambda) \rightarrow P$, $\tilde{e}_i,
\tilde{f}_i:\mathcal T'(\lambda) \rightarrow \mathcal T'(\lambda) \cup \{ 0 \}$ and $\varepsilon_i,
\varphi_i :\mathcal T'(\lambda) \rightarrow \mathbb{Z} \cup \{-\infty\}$  $(i \in I)
$ as follows:
\begin{equation}\label{eq:crystal on Yrlambda}
	\begin{aligned}
		&\wt(T)=\sum_{1\leq i\leq n+1}|T|_i\epsilon_{i},  \\
		&\tilde{e}_iT=\phi^{-1}(\tilde{f}_{n+1-i}\phi(T)),\quad\tilde{f}_iT=\phi^{-1}(\tilde{e}_{n+1-i}\phi(T)),\\
		&\varepsilon_i(T)=\varphi_{n+1-i}(\phi(T)),\quad
		\varphi_i(T)=\varepsilon_{n+1-i}(\phi(T)).
	\end{aligned}
\end{equation}

The following proposition demonstrates that the crystal structure on \(\mathcal{T}(\lambda)\) induces a corresponding crystal structure on \(\mathcal{T}'(\lambda)\).

\begin{Prop}\label{prop:crystal-Yprime}
	The set $\mathcal T'(\lambda)$ endowed with the maps in \eqref{eq:crystal on Yrlambda} forms a crystal of type $A_n$.
	Moreover, $\phi:\mathcal T'(\lambda)\to\mathcal T(\lambda)$ is a crystal anti-isomorphism intertwining
	\[
	\tilde e_i\ \longleftrightarrow\ \tilde f_{\,n+1-i},\qquad
	\tilde f_i\ \longleftrightarrow\ \tilde e_{\,n+1-i},\qquad
	\wt\ \longmapsto\ w_0\!\cdot\!\wt\quad(\epsilon_j\mapsto\epsilon_{n+2-j}).
	\]
\end{Prop}

\begin{proof}
	We verify the crystal axioms in Definition~\ref{def:abstract crystal}.
	
	\smallskip
	\noindent\textit{(i) The string length identity.}
	Recall $\langle \alpha_i^\vee,\epsilon_j\rangle=\delta_{i,j}-\delta_{i+1,j}$ and, for $T\in\mathcal T'(\lambda)$,
	\(|\phi(T)|_j=|T|_{\,n+2-j}\).
	Using the definitions in \eqref{eq:crystal on Yrlambda} and the identity
	$\varphi_k(\cdot)-\varepsilon_k(\cdot)=\langle \alpha_k^\vee,\wt(\cdot)\rangle$ on $\mathcal T(\lambda)$, we get
	\[
	\begin{aligned}
		\varphi_i(T)-\varepsilon_i(T)
		&=\varepsilon_{\,n+1-i}(\phi(T))-\varphi_{\,n+1-i}(\phi(T))
		=-\bigl(\varphi_{\,n+1-i}(\phi(T))-\varepsilon_{\,n+1-i}(\phi(T))\bigr)\\
		&=-\Big\langle \alpha_{\,n+1-i}^\vee,\ \sum_{j=1}^{n+1}|\phi(T)|_j\,\epsilon_j\Big\rangle
		= -\bigl(|\phi(T)|_{\,n+1-i}-|\phi(T)|_{\,n+2-i}\bigr)\\
		&= |T|_{i}-|T|_{i+1}
		=\Big\langle \alpha_i^\vee,\ \sum_{j=1}^{n+1}|T|_j\,\epsilon_j\Big\rangle
		=\langle \alpha_i^\vee,\wt(T)\rangle.
	\end{aligned}
	\]
	
	\smallskip
	\noindent\textit{(ii) and (iii) Weight change under $\tilde e_i,\tilde f_i$.}
	On $\mathcal T(\lambda)$, the operator $\tilde f_{\,n+1-i}$ changes exactly one letter
	\(n{+}1{-}i\) to \(n{+}2{-}i\) (when defined), hence
	\(|\tilde f_{\,n+1-i}(\phi(T))|_{n+1-i}=|\phi(T)|_{n+1-i}-1\),
	\(|\tilde f_{\,n+1-i}(\phi(T))|_{n+2-i}=|\phi(T)|_{n+2-i}+1\), and all other multiplicities are unchanged.
	Applying $\phi^{-1}$ we obtain
	\[
	\wt(\tilde e_iT)
	=\sum_{j=1}^{n+1}\bigl|\tilde f_{\,n+1-i}(\phi(T))\bigr|_{\,n+2-j}\,\epsilon_j
	=\wt(T)-\epsilon_{i+1}+\epsilon_i
	=\wt(T)+\alpha_i.
	\]
	The statement for $\tilde f_i$ is analogous and gives $\wt(\tilde f_iT)=\wt(T)-\alpha_i$ when $\tilde f_iT\neq0$.
	
	\smallskip
	\noindent\textit{(iv) and (v) Update of $\varepsilon_i,\varphi_i$ along arrows.}
	Using \eqref{eq:crystal on Yrlambda} and the corresponding rule on $\mathcal T(\lambda)$,
	\[
	\varepsilon_i(\tilde e_iT)
	=\varphi_{\,n+1-i}\!\bigl(\tilde f_{\,n+1-i}\phi(T)\bigr)
	=\varphi_{\,n+1-i}(\phi(T))-1
	=\varepsilon_i(T)-1,
	\]
	and
	\[
	\varphi_i(\tilde e_iT)
	=\varepsilon_{\,n+1-i}\!\bigl(\tilde f_{\,n+1-i}\phi(T)\bigr)
	=\varepsilon_{\,n+1-i}(\phi(T))+1
	=\varphi_i(T)+1.
	\]
	The relations for $\tilde f_i$ are proved in the same way.
	
	\smallskip
	\noindent\textit{(vi) Mutual inverses on their domains.}
	If $\tilde e_iT=\phi^{-1}\!\bigl(\tilde f_{\,n+1-i}\phi(T)\bigr)$, then
	\[
	\tilde f_i(\tilde e_iT)
	=\phi^{-1}\!\Bigl(\tilde e_{\,n+1-i}\phi\bigl(\phi^{-1}(\tilde f_{\,n+1-i}\phi(T))\bigr)\Bigr)
	=\phi^{-1}\!\bigl(\tilde e_{\,n+1-i}\tilde f_{\,n+1-i}\phi(T)\bigr)=T,
	\]
	and the converse is identical.
	
	\smallskip
	\noindent\textit{(vii) The $-\infty$ clause.}
	For highest weight crystals of type $A_n$, $\varepsilon_i,\varphi_i$ take values in $\Bbb Z_{\ge0}$ on
	$\mathcal T'(\lambda)$, so the clause about the value $-\infty$ is vacuous.
	
	\smallskip
	All axioms are satisfied, hence $\mathcal T'(\lambda)$ is an $A_n$-crystal with the stated structure.
\end{proof}

\begin{Ex}\label{ex:reverse}
Let $n=7$, and let  $\lambda=6\epsilon_1+4\epsilon_2+4\epsilon_3+2\epsilon_4+\epsilon_5+\epsilon_6$.
Figure \ref{fig:actione3} below illustrates  the action of $\tilde{e}_3$ on the reverse semi-standard Young tableau $T$.
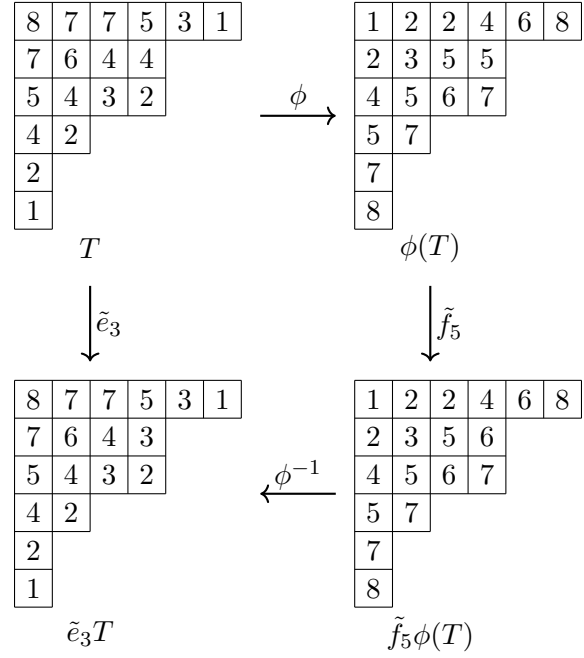
\begin{figure}[H] 
	\begin{tikzpicture}[scale=0.5]
\begin{scope}[shift={(0,0)}]	
	\draw (0,0)--(6,0);
	\draw (0,-1)--(6,-1);
	\draw (0,-2)--(4,-2);
	\draw (0,-3)--(4,-3);
	\draw (0,-4)--(2,-4);
	\draw (0,-5)--(1,-5);
	\draw (0,-6)--(1,-6);
\draw (0,0)--(0,-6);
\draw (1,0)--(1,-6);
\draw (2,0)--(2,-4);
\draw (3,0)--(3,-3);
\draw (4,0)--(4,-3);
\draw (5,0)--(5,-1);
\draw (6,0)--(6,-1);
\node at (0.5,-0.5){$8$};
\node at (1.5, -0.5){$7$};
\node at (2.5, -0.5){$7$};
\node at (3.5, -0.5){$5$};
\node at (4.5, -0.5){$3$};
\node at (5.5, -0.5){$1$};
\node at (0.5,-1.5){$7$};
\node at (1.5, -1.5){$6$};
\node at (2.5, -1.5){$4$};
\node at (3.5, -1.5){$4$};
\node at (0.5,-2.5){$5$};
\node at (1.5, -2.5){$4$};
\node at (2.5, -2.5){$3$};
\node at (3.5, -2.5){$2$};
\node at (0.5,-3.5){$4$};
\node at (1.5, -3.5){$2$};
\node at (0.5,-4.5){$2$};
\node at (0.5,-5.5){$1$};
\node at (2,-6.5){$T$};

\draw[thick,->] (6.5,-3)--(8.5,-3);
\node at (7.5,-2.5){$\phi$};

\draw[thick,->] (2,-7.5)--(2,-9.5);
\node at (2.5,-8.5){$\tilde{e}_3$};
\end{scope}	

\begin{scope}[shift={(9,0)}]	
	\draw (0,0)--(6,0);
	\draw (0,-1)--(6,-1);
	\draw (0,-2)--(4,-2);
	\draw (0,-3)--(4,-3);
	\draw (0,-4)--(2,-4);
	\draw (0,-5)--(1,-5);
	\draw (0,-6)--(1,-6);
\draw (0,0)--(0,-6);
\draw (1,0)--(1,-6);
\draw (2,0)--(2,-4);
\draw (3,0)--(3,-3);
\draw (4,0)--(4,-3);
\draw (5,0)--(5,-1);
\draw (6,0)--(6,-1);
\node at (0.5,-0.5){$1$};
\node at (1.5, -0.5){$2$};
\node at (2.5, -0.5){$2$};
\node at (3.5, -0.5){$4$};
\node at (4.5, -0.5){$6$};
\node at (5.5, -0.5){$8$};
\node at (0.5,-1.5){$2$};
\node at (1.5, -1.5){$3$};
\node at (2.5, -1.5){$5$};
\node at (3.5, -1.5){$5$};
\node at (0.5,-2.5){$4$};
\node at (1.5, -2.5){$5$};
\node at (2.5, -2.5){$6$};
\node at (3.5, -2.5){$7$};
\node at (0.5,-3.5){$5$};
\node at (1.5, -3.5){$7$};
\node at (0.5,-4.5){$7$};
\node at (0.5,-5.5){$8$};
\node at (2,-6.5){$\phi(T)$};
\draw[thick,->] (2,-7.5)--(2,-9.5);
\node at (2.5,-8.5){$\tilde{f}_5$};
\end{scope}

\begin{scope}[shift={(0,-10)}]	
	\draw (0,0)--(6,0);
	\draw (0,-1)--(6,-1);
	\draw (0,-2)--(4,-2);
	\draw (0,-3)--(4,-3);
	\draw (0,-4)--(2,-4);
	\draw (0,-5)--(1,-5);
	\draw (0,-6)--(1,-6);
\draw (0,0)--(0,-6);
\draw (1,0)--(1,-6);
\draw (2,0)--(2,-4);
\draw (3,0)--(3,-3);
\draw (4,0)--(4,-3);
\draw (5,0)--(5,-1);
\draw (6,0)--(6,-1);
\node at (0.5,-0.5){$8$};
\node at (1.5, -0.5){$7$};
\node at (2.5, -0.5){$7$};
\node at (3.5, -0.5){$5$};
\node at (4.5, -0.5){$3$};
\node at (5.5, -0.5){$1$};
\node at (0.5,-1.5){$7$};
\node at (1.5, -1.5){$6$};
\node at (2.5, -1.5){$4$};
\node at (3.5, -1.5){$3$};
\node at (0.5,-2.5){$5$};
\node at (1.5, -2.5){$4$};
\node at (2.5, -2.5){$3$};
\node at (3.5, -2.5){$2$};
\node at (0.5,-3.5){$4$};
\node at (1.5, -3.5){$2$};
\node at (0.5,-4.5){$2$};
\node at (0.5,-5.5){$1$};
\node at (2,-6.7){$\tilde{e}_3T$};

\draw[thick,<-] (6.5,-3)--(8.5,-3);
\node at (7.5,-2.5){$\phi^{-1}$};
\end{scope}

\begin{scope}[shift={(9,-10)}]	
	\draw (0,0)--(6,0);
	\draw (0,-1)--(6,-1);
	\draw (0,-2)--(4,-2);
	\draw (0,-3)--(4,-3);
	\draw (0,-4)--(2,-4);
	\draw (0,-5)--(1,-5);
	\draw (0,-6)--(1,-6);
\draw (0,0)--(0,-6);
\draw (1,0)--(1,-6);
\draw (2,0)--(2,-4);
\draw (3,0)--(3,-3);
\draw (4,0)--(4,-3);
\draw (5,0)--(5,-1);
\draw (6,0)--(6,-1);
\node at (0.5,-0.5){$1$};
\node at (1.5, -0.5){$2$};
\node at (2.5, -0.5){$2$};
\node at (3.5, -0.5){$4$};
\node at (4.5, -0.5){$6$};
\node at (5.5, -0.5){$8$};
\node at (0.5,-1.5){$2$};
\node at (1.5, -1.5){$3$};
\node at (2.5, -1.5){$5$};
\node at (3.5, -1.5){$6$};
\node at (0.5,-2.5){$4$};
\node at (1.5, -2.5){$5$};
\node at (2.5, -2.5){$6$};
\node at (3.5, -2.5){$7$};
\node at (0.5,-3.5){$5$};
\node at (1.5, -3.5){$7$};
\node at (0.5,-4.5){$7$};
\node at (0.5,-5.5){$8$};
\node at (2,-6.7){$\tilde{f}_5\phi(T)$};
\end{scope}		
	\end{tikzpicture}
	\vskip -3mm
	\caption{The action of $\tilde{e}_3$}\label{fig:actione3}
\end{figure} 
\end{Ex}

\begin{Lem}\label{lem:highest-lowest}
	Let $T^\lambda\in\mathcal T(\lambda)$ be the tableau whose $k$-th row is filled with $k$, and let
	$T_\lambda\in\mathcal T(\lambda)$ be the tableau whose $k$-th column, read from bottom to top, is
	$n+1,n,\dots,n+2-|\mathrm{col}_k(\lambda)|$. Then $T^\lambda$ is the highest weight element and
	$T_\lambda$ is the lowest weight element of the crystal $\mathcal T(\lambda)$, i.e.
	\[
	\tilde e_i T^\lambda=0\quad\text{and}\quad \tilde f_i T_\lambda=0\qquad(\forall\,i\in I).
	\]
\end{Lem}

\begin{proof}
	We use the standard Kashiwara--Nakashima (KN) reading: scan each row from right to left, and the rows from
	top to bottom (i.e. Far-Eastering reading). For the $i$-signature, write ``$+$'' for each entry $i$ and ``$-$'' for each entry $i{+}1$;
	then cancel all adjacent ``$(+, -)$'' pairs. After cancellation, the numbers of residual ``$-$'' and ``$+$'' are
	$\varepsilon_i$ and $\varphi_i$, respectively. Recall that $\tilde e_iT=0$ iff $\varepsilon_i(T)=0$, and
	$\tilde f_iT=0$ iff $\varphi_i(T)=0$.
	
	\smallskip
	\emph{Highest weight.} In $T^\lambda$, every $i$ appears exactly in row $i$, and every $i{+}1$ appears exactly
	in row $i{+}1$. Since we read the $i$-th row entirely before the $(i{+}1)$-st row, all ``$+$'' (from $i$) in
	the $i$-signature of $T^\lambda$ occur before all ``$-$'' (from $i{+}1$). Moreover, the length of row $i$ is
	at least that of row $i{+}1$, so every ``$-$'' is paired with a preceding ``$+$''. Hence no ``$-$'' remains,
	$\varepsilon_i(T^\lambda)=0$, and thus $\tilde e_iT^\lambda=0$ for all $i$.
	
	\smallskip
	\emph{Lowest weight.} In $T_\lambda$, each column is a consecutive string ending at $n{+}1$ and strictly
	increasing from top to bottom. Fix $i\in I$. If a column contains an $i$, then, by construction, it also
	contains an $(i{+}1)$ strictly \emph{below} that $i$. In the KN reading word, the $i$ from this column
	(contributing a ``$+$'') is encountered \emph{before} the $(i{+}1)$ from the same column (contributing a
	``$-$''). Therefore we can pair every such ``$+$'' with a later ``$-$'' in the same column, and these pairs are
	precisely of the canceling type ``$(+, -)$''. Consequently, after cancellation no ``$+$'' remains, i.e.\
	$\varphi_i(T_\lambda)=0$, and hence $\tilde f_iT_\lambda=0$ for all $i$.
	
	Thus $T^\lambda$ and $T_\lambda$ are the highest and lowest weight elements of $\mathcal T(\lambda)$, respectively.
\end{proof}

\begin{Ex}
Let us consider the values of \( n \) and \( \lambda \) as specified in Example \ref{ex:reverse}. In this context, the tableaux \( T^\lambda \) and \( T_\lambda \) are given in  Figure \ref{figTlambda}:
\begin{figure}[H] 
	\begin{tikzpicture}[scale=0.5]
	\begin{scope}[shift={(0,0)}]	
	\draw (0,0)--(6,0);
	\draw (0,-1)--(6,-1);
	\draw (0,-2)--(4,-2);
	\draw (0,-3)--(4,-3);
	\draw (0,-4)--(2,-4);
	\draw (0,-5)--(1,-5);
	\draw (0,-6)--(1,-6);
	\draw (0,0)--(0,-6);
	\draw (1,0)--(1,-6);
	\draw (2,0)--(2,-4);
	\draw (3,0)--(3,-3);
	\draw (4,0)--(4,-3);
	\draw (5,0)--(5,-1);
	\draw (6,0)--(6,-1);
	\node at (0.5,-0.5){$1$};
	\node at (1.5, -0.5){$1$};
	\node at (2.5, -0.5){$1$};
	\node at (3.5, -0.5){$1$};
	\node at (4.5, -0.5){$1$};
	\node at (5.5, -0.5){$1$};
	\node at (0.5,-1.5){$2$};
	\node at (1.5, -1.5){$2$};
	\node at (2.5, -1.5){$2$};
	\node at (3.5, -1.5){$2$};
	\node at (0.5,-2.5){$3$};
	\node at (1.5, -2.5){$3$};
	\node at (2.5, -2.5){$3$};
	\node at (3.5, -2.5){$3$};
	\node at (0.5,-3.5){$4$};
	\node at (1.5, -3.5){$4$};
	\node at (0.5,-4.5){$5$};
	\node at (0.5,-5.5){$6$};
	\node at (2,-6.5){$T^\lambda$};

	\end{scope}	
	
	\begin{scope}[shift={(9,0)}]	
	\draw (0,0)--(6,0);
	\draw (0,-1)--(6,-1);
	\draw (0,-2)--(4,-2);
	\draw (0,-3)--(4,-3);
	\draw (0,-4)--(2,-4);
	\draw (0,-5)--(1,-5);
	\draw (0,-6)--(1,-6);
	\draw (0,0)--(0,-6);
	\draw (1,0)--(1,-6);
	\draw (2,0)--(2,-4);
	\draw (3,0)--(3,-3);
	\draw (4,0)--(4,-3);
	\draw (5,0)--(5,-1);
	\draw (6,0)--(6,-1);
	\node at (0.5,-0.5){$3$};
	\node at (1.5, -0.5){$5$};
	\node at (2.5, -0.5){$6$};
	\node at (3.5, -0.5){$6$};
	\node at (4.5, -0.5){$8$};
	\node at (5.5, -0.5){$8$};
	\node at (0.5,-1.5){$4$};
	\node at (1.5, -1.5){$6$};
	\node at (2.5, -1.5){$7$};
	\node at (3.5, -1.5){$7$};
	\node at (0.5,-2.5){$5$};
	\node at (1.5, -2.5){$7$};
	\node at (2.5, -2.5){$8$};
	\node at (3.5, -2.5){$8$};
	\node at (0.5,-3.5){$6$};
	\node at (1.5, -3.5){$8$};
	\node at (0.5,-4.5){$7$};
	\node at (0.5,-5.5){$8$};
	\node at (2,-6.5){$T_\lambda$};
	\end{scope}		
	\end{tikzpicture}
	\vskip -3mm
	\caption{The tableaux corresponding to highest and lowest weight vectors}\label{figTlambda}
\end{figure} 
\end{Ex}

We have the following proposition:

\begin{Prop}\label{prop:rholambda}
	There exists an involution $\rho_\lambda:\mathcal T(\lambda)\to\mathcal T(\lambda)$ such that
	\[
	\rho_\lambda\!\bigl(\tilde f_{i_k}\cdots\tilde f_{i_1}T^\lambda\bigr)
	=\tilde e_{\,n+1-i_k}\cdots\tilde e_{\,n+1-i_1}T_\lambda .
	\]
\end{Prop}

\begin{proof}
	Let $\phi:\mathcal T'(\lambda)\to\mathcal T(\lambda)$ be the bijection from Proposition~\ref{prop:crystal-Yprime}, and
	let $\tilde e_i',\tilde f_i'$ be the crystal operators on $\mathcal T'(\lambda)$ defined by
	\[
	\tilde e_i'=\phi^{-1}\tilde f_{\,n+1-i}\phi,\qquad
	\tilde f_i'=\phi^{-1}\tilde e_{\,n+1-i}\phi .
	\]
	By Lemma~\ref{lem:highest-lowest}, $T^\lambda$ (resp.\ $T_\lambda$) is the highest (resp.\ lowest) element of
	$\mathcal T(\lambda)$. Hence $\tilde f_{\,n+1-i}T_\lambda=0$ for all $i$, so
	$\tilde e_i'\bigl(\phi^{-1}(T_\lambda)\bigr)=\phi^{-1}\tilde f_{\,n+1-i}T_\lambda=0$ for all $i$; that is,
	$\phi^{-1}(T_\lambda)$ is the highest element of the highest weight crystal $\bigl(\mathcal T'(\lambda);\tilde e_i',\tilde f_i'\bigr)$.
	
	\smallskip
	\emph{Well-definedness of $\rho_\lambda$.}
	For $T\in\mathcal T(\lambda)$ choose indices $i_1,\dots,i_k$ with
	$T=\tilde f_{i_k}\cdots\tilde f_{i_1}T^\lambda$ and set
	\begin{equation}\label{eq:rho-def}
		\rho_\lambda(T)\ :=\ \phi\bigl(\tilde f'_{i_k}\cdots\tilde f'_{i_1}\,\phi^{-1}(T_\lambda)\bigr).
	\end{equation}
	Because $\bigl(\mathcal T'(\lambda);\tilde e_i',\tilde f_i'\bigr)$ is a connected highest weight crystal with highest element
	$\phi^{-1}(T_\lambda)$, the right-hand side of \eqref{eq:rho-def} is independent of the chosen factorization of $T$
	(the crystal local relations for type $A_n$ hold equally for the primed operators). Thus $\rho_\lambda$ is well-defined.
	Taking $k=0$ shows $\rho_\lambda(T^\lambda)=T_\lambda$.
	
	\smallskip
	\emph{Intertwining with $\tilde f$.}
	From $\phi\tilde f'_i=\tilde e_{\,n+1-i}\phi$ we get, for any $i$ and any $T$,
	\[
	\rho_\lambda(\tilde f_iT)
	=\phi\!\left(\tilde f'_i\,\tilde f'_{i_k}\cdots\tilde f'_{i_1}\phi^{-1}(T_\lambda)\right)
	=\tilde e_{\,n+1-i}\,\phi\!\left(\tilde f'_{i_k}\cdots\tilde f'_{i_1}\phi^{-1}(T_\lambda)\right)
	=\tilde e_{\,n+1-i}\,\rho_\lambda(T).
	\]
	Iterating this identity yields exactly the displayed formula in the statement:
	\[
	\rho_\lambda\!\bigl(\tilde f_{i_k}\cdots\tilde f_{i_1}T^\lambda\bigr)
	=\tilde e_{\,n+1-i_k}\cdots\tilde e_{\,n+1-i_1}T_\lambda .
	\]
	
	\smallskip
	\emph{Involutivity.}
	Consider $\rho_\lambda^2$. We have $\rho_\lambda(T^\lambda)=T_\lambda$, and by the previous paragraph
	$\rho_\lambda\circ\tilde f_i=\tilde e_{\,n+1-i}\circ\rho_\lambda$. Hence
	\[
	\rho_\lambda^2\circ \tilde f_i
	=\rho_\lambda\circ \tilde e_{\,n+1-i}\circ \rho_\lambda
	=\tilde f_i\circ \rho_\lambda^2\qquad(\forall i),
	\]
and $
\tilde e_{i}\,\rho_\lambda(T_\lambda)=\rho_\lambda(\tilde f_{n+1-i}T_\lambda)=0$.

Hence $\rho_\lambda(T_\lambda)$ is annihilated by all $\tilde e_i$, so it is the highest weight element of the
connected highest weight crystal $\mathcal T(\lambda)$; by uniqueness, $\rho_\lambda(T_\lambda)=T^\lambda$.
Therefore $\rho_\lambda^{\,2}(T^\lambda)=T^{\lambda}$.

Since every element of the connected highest weight crystal
	$\mathcal T(\lambda)$ is of the form $\tilde f_{i_k}\cdots\tilde f_{i_1}T^\lambda$, it follows that
	$\rho_\lambda^2$ fixes all elements; hence $\rho_\lambda^2=\mathrm{id}$. Therefore $\rho_\lambda$ is an involution.
\end{proof}

\begin{Rmk}
The involution \( \rho_\lambda \) defined on \( \mathcal{T}(\lambda) \) coincides with the Schützenberger involution on \( \mathcal{T}(\lambda) \). More precisely, for any \( T \in \mathcal{T}(\lambda) \), let \( (\phi^{-1}(T))^{\#} \) denote the skew Young tableau obtained by rotating \( \phi^{-1}(T) \) by \( 180^\circ \) in the plane. Applying the \emph{jeu de taquin} procedure to \( \phi^{-1}(T)^{\#} \) produces a tableau \( 	
\evac(T):=\jdt\!\bigl((\phi^{-1}(T))^{\#}\bigr)
 \), which corresponds to \( \rho_\lambda(T) \) (cf.~\cite[Proposition~2.87]{Shi05}).
\end{Rmk}

\begin{Ex}
Fix $n=3$. The left graph in the following Figure \ref{graphwith180} illustrates the crystal graph of
$\mathcal B(\lambda)$
with highest weight $\lambda=2\epsilon_1+\epsilon_2$. The right graph is obtained by rotating the left graph by $180^\circ$.
\begin{figure}[H]
	\begin{minipage}[t]{0.51\linewidth}
		\centering
		\includegraphics[width=0.76\textwidth]{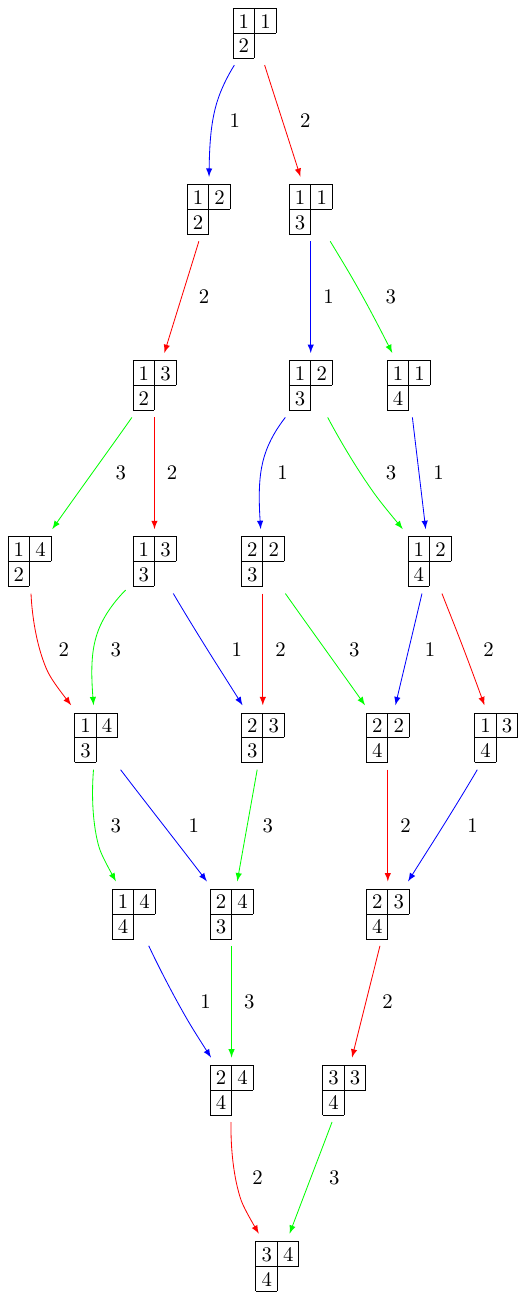}
	\end{minipage}%
	\begin{minipage}[t]{0.51\linewidth}
		\centering
	\includegraphics[width=0.76\textwidth]{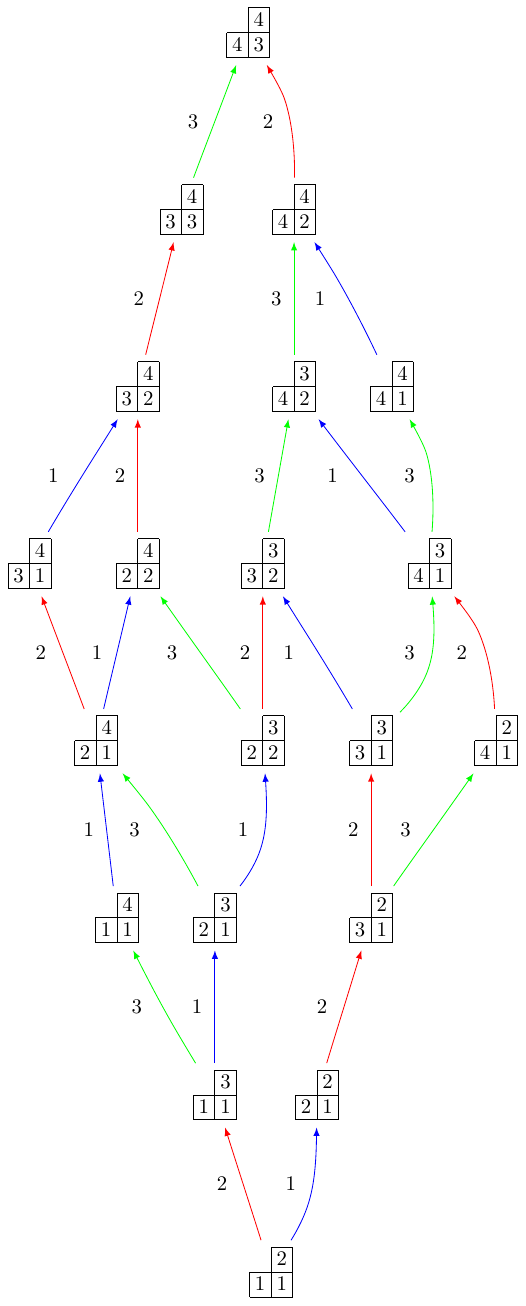}
	\end{minipage}
\vskip 1mm
\caption{The crystal graph and the crystal graph obtained by rotation}\label{graphwith180}
\end{figure}

In the directed graph on the right, we perform the following sequence of transformations:
\begin{enumerate}
	\item Reverse the direction of each arrow;
	\item Relabel each arrow by replacing its label with \( 4 \) minus its original value;
	\item Replace each entry in the rotated Young tableaux with \( 5 \) minus its original value, thereby producing skew Young tableaux;
	\item Apply the jeu de taquin procedure to each skew tableau to obtain a semi-standard Young tableau.
\end{enumerate}
Upon completing steps (1)--(4), we recover the crystal graph shown on the left.
\end{Ex}

\section{
Young tableau description for the polyhedral realizations of 	$\mathcal B(\infty)$ and $\mathcal B(\lambda)$}\label{sec:tableau_polyhedral_Binfty_Blambda}

In this section, the combinatorics of Young tableaux is used to give an explicit combinatorial description of the polyhedral realizations of $\mathcal B(\infty)$ and $\mathcal B(\lambda)$.
Here, we review only the explicit polyhedral realizations for $\mathcal B(\infty)$ and $\mathcal B(\lambda)$ in type $A_n$; the general case is presented in Appendix~\ref{sec:poly_general}.

\vskip 2mm

We choose a periodic sequence $\iota$ as follows:
\begin{equation}\label{eq:periodic sequence}
	\iota=\cdots,n,\cdots,2,1,\cdots,n,\cdots,2,1,\cdots,n,\cdots,2,1.
\end{equation}

There is a bijection
$\mathbb Z_{\geq 1}\times I\to \mathbb Z_{\geq 1}$, which is given by
$(j,i)\mapsto (j-1)n+i$. Therefore, we can identify 
$x_k\in\mathbb Z$ $(k>1)$ with $x_{(j-1)n+i}$. For convenience, we define $x_j^{(i)}:=x_{(j-1)n+i}$ for $j\geq 1$ and $i\in I$.

\begin{Thm}\cite[Theorem 5.1]{NZ97}\label{thm:polyhedral Binfty A}
	The crystal $\mathcal B(\infty)$ is realized as the following set 	
	\begin{equation}\label{eq:poly_Binfty}
		\Sigma_\iota=
		\left\{
		\overrightarrow{x}=(x_j^{(i)})_{i,j\ge 1}\ \middle|\ 
		\begin{array}{l}
			x_1^{(i)}\ge x_2^{(i-1)}\ge \cdots \ge x_i^{(1)}\ge 0,\quad 1\le i\le n,\\[2pt]
			x_j^{(i)}=0,\quad i+j>n+1
		\end{array}
		\right\},
	\end{equation}
	with crystal data $\wt, \tilde f_i, \tilde e_i,\varepsilon_i, \varphi_i$  given in \eqref{eq:crystal operator on Zinfty}.
\end{Thm}

\begin{Thm}\cite[Theorem 6.1]{Na99}\label{thm:polyhedral A}
For any dominant weight $\lambda\in P^+$, the crystal $\mathcal B(\lambda)$ is realized as the following set 	
	\begin{equation}\label{eq:poly_Blambda}
		\Sigma_\iota[\lambda]=
		\left\{
		\overrightarrow{x}=(x_j^{(i)})_{i,j\ge 1}\ \middle|\ 
		\begin{array}{l}
			x_1^{(i)}\ge x_2^{(i-1)}\ge \cdots \ge x_i^{(1)}\ge 0,\quad 1\le i\le n,\\[2pt]
			x_j^{(i)}=0,\quad i+j>n+1,\\[2pt]
			\lambda_i-\lambda_{i+1}\ge x_j^{(i-j+1)}-x_j^{(i-j)},\quad 1\le j\le i\le n
		\end{array}
		\right\},
	\end{equation}
	with crystal data $\wt, \tilde f_i, \tilde e_i,\varepsilon_i, \varphi_i$  given in \eqref{eq:crystal operator}.	
\end{Thm}

\subsection{The case $\mathcal B(\infty)$}\label{sec:YTAFPR}

We define the map
\begin{equation}\label{eq:TinftytoSigmainfty}
	\psi_\infty:\mathcal  T'(\infty)\longrightarrow \Sigma_\iota,\quad T\mapsto \overrightarrow{x}_T
\end{equation}
by setting 
\begin{equation*}
	x_j^{(i)}=
	\begin{cases}
		\sum_{k=1}^iz_{n+2-i-j}^k(T)& 1\leq i\leq n,\ 2\leq i+j\leq n+1,\\
		0&i+j> n+1.
	\end{cases}
\end{equation*}
for any $T\in \mathcal T'(\infty)$, where the tableau \( T \) is determined by $z_{r}^s$ $(1\leq r\leq n, 1\leq s\leq n+1-r)$.
\begin{Rmk}
	Since $x_j^{(i)}-x_{j+1}^{(i-1)}=z_{n+2-i-j}^i\geq 0$, it follows from  \eqref{eq:poly_Binfty} that the map $\psi_\infty$ in \eqref{eq:TinftytoSigmainfty} is well-defined.
\end{Rmk}

\begin{Thm}\label{thm:psiinfty}
	The map $\psi_\infty$ is a crystal isomorphism.
\end{Thm}
\begin{proof}
	
	It is straightforward to verify that \( \psi_\infty \) is bijective. In what follows, we show that \( \psi_\infty \) preserves the crystal structure.
	
	\vskip 1mm
	
	From the definition of $\wt$ in \eqref{eq:crystal on Yrinfty}, it follows that
	\begin{equation}\label{eq:wtrhoT}
		\begin{aligned}
			\wt(\psi_\infty(T))=&\wt(\overrightarrow{x}_T)	
			=-\sum_{1<i+j\leq	 n+1}x_j^{(i)}\alpha_{i}=-\sum_{i=1}^n(\sum_{j=1}^{n+1-i}(\sum_{k=1}^iz_{n+2-i-j}^k))\alpha_i\\
			=&-\sum_{i=1}^n(\sum_{j=1}^{n+1-i}(\sum_{k=1}^iz_{j}^k))\alpha_i=\wt(T).
		\end{aligned}
	\end{equation}

	By \eqref{eq:epsilonTinfty} and Lemma \ref{lem:sigmaixvalue}, we obtain
	\begin{equation}\label{eq:epsilonirhoT}
		\begin{aligned}
			&\varepsilon_i(\psi_\infty T)=\varepsilon_i(\overrightarrow{x}_T)\\
			=&\max_{j\geq 1}\{x_j^{(i)}-x_j^{(i+1)}+2\sum_{l=j+1}^{n+1-i}x_l^{(i)}-\sum_{l=j+1}^{n+2-i}x_l^{(i-1)}-\sum_{l=j+1}^{n-i}x_l^{(i+1)}\}\\
			=&\max_{j\geq 1}\{
			\sum_{k=1}^iz_{n+2-i-j}^{k}-\sum_{k=1}^{i+1}z_{n+1-i-j}^{k}+2\sum_{l=j+1}^{n+1-i}\sum_{k=1}^i z_{n+2-i-l}^{k}\\
			&\qquad-\sum_{l={j+1}}^{n+2-i}\sum_{k=1}^{i-1}z_{n+3-l-i}^{k}-\sum_{l=j+1}^{n-i}\sum_{k=1}^{i+1}z_{n+1-i-l}^{k}
			\}\\
			=&\max_{j\geq 1}\{\sum_{l=1}^{n+2-i-j}\sum_{k=1}^i z_{l}^{k}+\sum_{l=1}^{n+1-i-j}\sum_{k=1}^i z_{l}^{k}
			-\sum_{l=1}^{n+2-i-j}\sum_{k=1}^{i-1}z_{l}^{k}-\sum_{l=1}^{n+1-i-j}\sum_{k=1}^{i+1}z_{l}^{k}\}\\
			=&\max_{j\geq 1}\{{\sum_{l=1}^{n+2-i-j}}(z_l^{i}-z_{l-1}^{i+1})\}
			=\max_{1\leq j\leq n+1-i}\{{\sum_{l=1}^{n+2-i-j}}(z_l^{i}-z_{l-1}^{i+1})\}\\
			=&\max_{1\leq j\leq n+1-i}\{\sum_{k=1}^j(y_k^{n+2-i}-y_{k-1}^{n+1-i})\}
			=\varepsilon_{n+1-i}(\eta T)=\varepsilon_i(T).
		\end{aligned}
	\end{equation}

	As a consequence of \eqref{eq:wtrhoT}--\eqref{eq:epsilonirhoT}, the identity $\varphi_i(\psi_\infty T)=\varphi_i(T)$ holds naturally.
	
	\vskip 1mm
	
	We now proceed to prove that \( \tilde{e}_i (\psi_\infty T) = \psi_\infty (\tilde{e}_i T) \).
	From \eqref{eq:epsilonirhoT}, it follows that
	$$ \varepsilon_{n+1-i}(\eta T) = \max_{1\leq j\leq n+1-i}\{{\sum_{l=1}^{n+2-i-j}}(z_l^{i}-z_{l-1}^{i+1})\}.$$
	
	Suppose that 
	\begin{equation}\label{eq:varepsilonetaT}
		\varepsilon_{n+1-i}(\eta T) = {\sum_{l=1}^{n+2-i-j}}(y_l^{n+2-i}-y_{l-1}^{n+1-i})>0.
	\end{equation}
	
	For \( j \neq 1 \), the action of \( \tilde{e}_{n+1-i} \) on \( \eta T \) changes 
	$$ 
	y_{n+2-i-j}^{n+2-i} \text { to }  y_{n+2-i-j}^{n+2-i}-1 ,\   y_{n+2-i-j}^{n+1-i}  \text{ to }  y_{n+2-i-j}^{n+1-i}+1 
	$$ 
	and leaves all other \( y_i^j \) in \( \eta T \) unchanged.
	If  \( j = 1 \), then the action of \( \tilde{e}_{n+1-i} \) on \( \eta T \) changes \( y_{n+1-i}^{n+2-i} \) to \( y_{n+1-i}^{n+2-i}-1 \),  and leaves all other \( y_i^j \) in \( \eta T \) unchanged. 
	
	\vskip 1mm
	
	From the definition of  \( \psi_\infty \) in \eqref{eq:TinftytoSigmainfty}, it follows that
	\begin{equation}\label{eq:xT-XeT}
		\begin{aligned}
			(\overrightarrow{x}_T)_{(r-1)n+s}-(\overrightarrow{x}_{\tilde{e}_iT})_{(r-1)n+s}=
			\begin{cases}
				1& \text{if}\ (r,s)=(j,i),\\
				0& \text{if}\ (r,s)\neq (j,i).
			\end{cases}
		\end{aligned}
	\end{equation}

	By \eqref{eq:epsilonirhoT} and the definitions of $M^{(i)}(\overrightarrow{x})$, $\varepsilon_i(\overrightarrow{x})$ in \eqref{eq:sigmai and Mi}--\eqref{eq:crystal operator on Zinfty}, we obtain that $\max M^{(i)}(\overrightarrow{x}_T)=(j-1)n+i$. Hence, we conclude that
	\begin{equation}\label{eq:xT-exT}
		\begin{aligned}
			(\overrightarrow{x}_T)_{(r-1)n+s}-(\tilde{e}_i\overrightarrow{x}_{T})_{(r-1)n+s}=
			\begin{cases}
				1& \text{if}\ (r,s)=(j,i),\\
				0& \text{if}\ (r,s)\neq (j,i).
			\end{cases}
		\end{aligned}
	\end{equation}
	
	The formulas in \eqref{eq:xT-XeT}--\eqref{eq:xT-exT} imply that \(  \tilde{e}_i\overrightarrow{x}_{T}=\overrightarrow{x}_{\tilde{e}_iT}\). 
	
	\vskip 1mm
	
	By the construction of RMLT, for any \( T \in \mathcal{T}'(\infty) \) and \( i \in I \), we have \( \tilde{f}_i T = T' \in \mathcal{T}'(\infty) \), which implies $\tilde{e}_iT'=T$. Since \( \tilde{e}_i \psi_\infty T' = \psi_\infty \tilde{e}_i T' \), it follows that \(  \psi_\infty T' =\tilde{f}_i( \psi_\infty \tilde{e}_i T') \).  Therefore, we conclude that $\psi_\infty(\tilde{f}_i T)=\tilde{f}_i (\psi_\infty T)$.
\end{proof}

From Proposition \ref{prop:infcrystal}--\ref{prop:rhoinfty-via-eta} and Theorem~\ref{thm:psiinfty}, we obtain the following corollary.

\begin{Cor}
	Define
	\[
	\rho_{\infty}^{\mathrm{poly}}
	\ :=\
	\psi_\infty\ \circ\ \eta^{-1}\ \circ\ \rho_\infty\ \circ\ \eta\ \circ\ \psi_\infty^{-1}
	\ :\ \Sigma_\iota \longrightarrow \Sigma_\iota.
	\]
	Then $\rho_{\infty}^{\mathrm{poly}}$ is an involution on $\Sigma_\iota$. Moreover, for all $i\in I$ and
	$\vec x\in\Sigma_\iota$,
	\[
	\rho_{\infty}^{\mathrm{poly}}\circ \tilde f_i
	\;=\;
	\tilde f_{\,n+1-i}\circ \rho_{\infty}^{\mathrm{poly}},
	\qquad
	\rho_{\infty}^{\mathrm{poly}}\circ \tilde e_i
	\;=\;
	\tilde e_{\,n+1-i}\circ \rho_{\infty}^{\mathrm{poly}},
	\qquad
	\wt\bigl(\rho_{\infty}^{\mathrm{poly}}(\vec x)\bigr)
	\;=\;
	\tau\bigl(\wt(\vec x)\bigr).
	\]
\end{Cor}

\vskip 2mm

\subsection{The case $\mathcal B(\lambda)$}

By Theorem \ref{thm:polyhedral A}, for any $\overrightarrow{x}\in \Sigma_\iota[\lambda]$, we define the sequence 
\begin{equation}\label{eq:seq Lambda}
{\Lambda}_i(\overrightarrow{x})=(x_k^{(i-1)}+\lambda_{k+i-1})_{1\leq k\leq n+2-i}
\end{equation}
for each $i\in\{1,2,\dots,n+1\}$.

\begin{Lem}\label{eq:inequality of poly}
For any two consecutive  sequences $\Lambda_i(\overrightarrow{x})$ and $\Lambda_{i+1}(\overrightarrow{x})$$(1\leq i\leq n)$ in \eqref{eq:seq Lambda}, the following inequalities hold:
\begin{equation}\label{eq：tableau inequality}
x_k^{(i-1)}+\lambda_{k+i-1}\geq x_k^{(i)}+\lambda_{k+i}\geq x_{k+1}^{(i-1)}+\lambda_{k+i}
\end{equation}
where $1\leq k\leq n+1-i$.
\end{Lem}

\begin{proof}
It follows from \eqref{eq:poly_Blambda} that $\lambda_{k+i-1}-\lambda_{(k+i-1)+1}\geq x_j^{(k+i-1-j+1)}-x_j^{(k+i-1-j)}$. By setting 
$j=k$, we obtain the first inequality in \eqref{eq：tableau inequality}. The second inequality also follows immediately from  \eqref{eq:poly_Blambda}.
\end{proof}

Let $\Lambda(\overrightarrow{x})=(\Lambda_i(\overrightarrow{x}))_{1\leq i\leq n+1}$.
By Lemma \ref{eq:inequality of poly}, each sequence $\Lambda_i(\overrightarrow{x})$ can be interpreted  as a Young diagram, and the skew Young diagram $\Lambda_i(\overrightarrow{x})/\Lambda_{i+1}(\overrightarrow{x})$ contains no adjacent boxes within the same column.

\vskip 1mm

For each \(i \in \{1, 2, \dots, n+1\}\), fill every box in the skew Young diagram \(\Lambda_i(\overrightarrow{x}) / \Lambda_{i+1}(\overrightarrow{x})\) with the entry \(i\). By stacking these filled skew diagrams together, we obtain a reverse semi-standard Young tableau, denoted by \(T_{\overrightarrow{x}}\).  
This construction defines a map
\begin{equation}\label{eq:psi}
\psi_\lambda:\Sigma_\iota[\lambda]\longrightarrow \mathcal T'(\lambda) \ \text{by setting}\ \overrightarrow{x}\mapsto T_{\overrightarrow{x}}.
\end{equation}

\begin{Rmk}
If $\overrightarrow{x}=\mathbf 0$, then we denote $T_{\overrightarrow{x}}=\mathbf 0$.
\end{Rmk}

If there is no risk of confusion, we will denote \(\Lambda_i(\overrightarrow{x})\) by \(\Lambda_i\) in the following text. Let \(\Lambda_i^{(j)}(\overrightarrow{x})\) denote the \(j\)-th entry of \(\Lambda_i(\overrightarrow{x})\), and let \(|T_{\overrightarrow{x}}|_i^{(j)}\) denote the number of boxes labeled \(i\) in the \(j\)-th row of the tableau \(T_{\overrightarrow{x}}\).

%

\begin{Ex}
We consider the case of $n=9$ and $\lambda=8\epsilon_1+6\epsilon_2+4\epsilon_3+3\epsilon_4+3\epsilon_5+2\epsilon_6+\epsilon_7+\epsilon_8$.	Let $\overrightarrow{x}=(\cdots, x_2, x_1)$ be given in the following form:
\begin{align*}
&(x_1^{(1)}, x_2^{(1)}, x_3^{(1)}, x_4^{(1)}, x_5^{(1)}, x_6^{(1)}, x_7^{(1)}, x_8^{(1)}, x_9^{(1)})=(1,1,1,0,0,0,0,0,0),\\
&(x_1^{(2)}, x_2^{(2)}, x_3^{(2)}, x_4^{(2)}, x_5^{(2)}, x_6^{(2)}, x_7^{(2)}, x_8^{(2)})=(2,2,1,1,0,0,1,0),\\
&(x_1^{(3)}, x_2^{(3)}, x_3^{(3)}, x_4^{(3)}, x_5^{(3)}, x_6^{(3)}, x_7^{(3)})
=(2,2,2,1,0,1,0),\\
&(x_1^{(4)}, x_2^{(4)}, x_3^{(4)}, x_4^{(4)}, x_5^{(4)}, x_6^{(4)})=(2,2,3,0,1,0),\ (x_1^{(5)}, x_2^{(5)}, x_3^{(5)}, x_4^{(5)}, x_5^{(5)})=(3,3,1,1,0),\\
&(x_1^{(6)}, x_2^{(6)}, x_3^{(6)}, x_4^{(6)})=(3,2,1,0),\ (x_1^{(7)}, x_2^{(7)}, x_3^{(7)})=(2,3,0),\\
&(x_1^{(8)}, x_2^{(8)})=(3,1),\ (x_1^{(9)})=(2),
\end{align*}
and $x_j^{(i)}=0$ for $i+j>10$. Then we have
\begin{align*}
\Lambda_1&=(8,6,4,3,3,2,1,1),\ \Lambda_2=(7,5,4,3,2,1,1),\ \Lambda_3=(6,5,4,3,1,1,1),\\
\Lambda_4&=(5,5,4,2,1,1),\ \Lambda_5=(5,4,4,1,1),\ \Lambda_6=(5,4,2,1),\ \Lambda_7=(4,3,1,0),\\
\Lambda_8&=(3,3),\ \Lambda_9=(3,1),\ \Lambda_{10}=(2).
\end{align*}
The skew-Young tableaux corresponding to the sequence $$(\Lambda_1/\Lambda_2,\  \Lambda_2/\Lambda_3, \ \Lambda_3/\Lambda_4,\  \Lambda_4/\Lambda_5, \ \Lambda_5/\Lambda_6,\  \Lambda_6/\Lambda_7, \ \Lambda_7/\Lambda_8,\  \Lambda_8/\Lambda_9, \ \Lambda_9/\Lambda_{10},\  \Lambda_{10})$$
are listed as follows:
\begin{figure}[H] 
	\begin{tikzpicture}[scale=0.48]
		\begin{scope}[shift={(0,0)}]	
		\draw (0,0)--(8,0);
		\draw (0,-1)--(8,-1);
		\draw (0,-2)--(6,-2);
		\draw (0,-3)--(4,-3);
		\draw (0,-4)--(3,-4);
		\draw (0,-5)--(3,-5);
		\draw (0,-6)--(2,-6);
		\draw (0,-7)--(1,-7);	
		\draw (0,-8)--(1,-8);	
		\draw (0,0)--(0,-8);
		\draw (1,0)--(1,-8);
		\draw (2,0)--(2,-6);
		\draw (3,0)--(3,-5);
		\draw (4,0)--(4,-3);
		\draw (5,0)--(5,-2);
		\draw (6,0)--(6,-2);
		\draw (7,0)--(7,-1);
		\draw (8,0)--(8,-1);
\node at (7.5, -0.5){$1$};
	\node at (5.5, -1.5){$1$};
	\node at (2.5,-4.5){$1$};
	\node at (1.5,-5.5){$1$};
		\node at (0.5,-7.5){$1$};
\end{scope}		

\begin{scope}[shift={(9,0)}]	
	\draw (0,0)--(7,0);
	\draw (0,-1)--(7,-1);
	\draw (0,-2)--(5,-2);
	\draw (0,-3)--(4,-3);
	\draw (0,-4)--(3,-4);
	\draw (0,-5)--(2,-5);
	\draw (0,-6)--(1,-6);
	\draw (0,-7)--(1,-7);	
	
	\draw (0,0)--(0,-7);
	\draw (1,0)--(1,-7);
	\draw (2,0)--(2,-5);
	\draw (3,0)--(3,-4);
	\draw (4,0)--(4,-3);
	\draw (5,0)--(5,-2);
	\draw (6,0)--(6,-1);
	\draw (7,0)--(7,-1);
\node at (6.5, -0.5){$2$};
\node at (1.5,-4.5){$2$};
\end{scope}	

\begin{scope}[shift={(17,0)}]	
	\draw (0,0)--(6,0);
	\draw (0,-1)--(6,-1);
	\draw (0,-2)--(5,-2);
	\draw (0,-3)--(4,-3);
	\draw (0,-4)--(3,-4);
	\draw (0,-5)--(1,-5);
	\draw (0,-6)--(1,-6);
	\draw (0,-7)--(1,-7);	
	
	\draw (0,0)--(0,-7);
	\draw (1,0)--(1,-7);
	\draw (2,0)--(2,-4);
	\draw (3,0)--(3,-4);
	\draw (4,0)--(4,-3);
	\draw (5,0)--(5,-2);
	\draw (6,0)--(6,-1);
\node at (5.5, -0.5){$3$};
\node at (2.5, -3.5){$3$};
\node at (0.5,-6.5){$3$};
\end{scope}	

\begin{scope}[shift={(24,0)}]	
	\draw (0,0)--(5,0);
	\draw (0,-1)--(5,-1);
	\draw (0,-2)--(5,-2);
	\draw (0,-3)--(4,-3);
	\draw (0,-4)--(2,-4);
	\draw (0,-5)--(1,-5);
	\draw (0,-6)--(1,-6);

	\draw (0,0)--(0,-6);
	\draw (1,0)--(1,-6);
	\draw (2,0)--(2,-4);
	\draw (3,0)--(3,-3);
	\draw (4,0)--(4,-3);
	\draw (5,0)--(5,-2);
\node at (4.5, -1.5){$4$};
\node at (1.5, -3.5){$4$};
\node at (0.5,-5.5){$4$};
\end{scope}	

\begin{scope}[shift={(0,-10)}]	
	\draw (0,0)--(5,0);
	\draw (0,-1)--(5,-1);
	\draw (0,-2)--(4,-2);
	\draw (0,-3)--(4,-3);
	\draw (0,-4)--(1,-4);
	\draw (0,-5)--(1,-5);
	\draw (0,-5)--(1,-5);

	\draw (0,0)--(0,-5);
	\draw (1,0)--(1,-5);
	\draw (2,0)--(2,-3);
	\draw (3,0)--(3,-3);
	\draw (4,0)--(4,-3);
	\draw (5,0)--(5,-1);
\node at (2.5, -2.5){$5$};
\node at (3.5, -2.5){$5$};
\node at (0.5,-4.5){$5$};
\end{scope}	

\begin{scope}[shift={(6,-10)}]	
	\draw (0,0)--(5,0);
	\draw (0,-1)--(5,-1);
	\draw (0,-2)--(4,-2);
	\draw (0,-3)--(2,-3);
	\draw (0,-4)--(1,-4);
	\draw (0,0)--(0,-4);
	\draw (1,0)--(1,-4);
	\draw (2,0)--(2,-3);
	\draw (3,0)--(3,-2);
	\draw (4,0)--(4,-2);
	\draw (5,0)--(5,-1);
\node at (4.5, -0.5){$6$};
\node at (3.5, -1.5){$6$};
\node at (1.5, -2.5){$6$};
\node at (0.5,-3.5){$6$};
\end{scope}	

\begin{scope}[shift={(12,-10)}]	
	\draw (0,0)--(4,0);
	\draw (0,-1)--(4,-1);
	\draw (0,-2)--(3,-2);
	\draw (0,-3)--(1,-3);

	\draw (0,0)--(0,-3);
	\draw (1,0)--(1,-3);
	\draw (2,0)--(2,-2);
	\draw (3,0)--(3,-2);
	\draw (4,0)--(4,-1);
\node at (3.5, -0.5){$7$};
\node at (0.5,-2.5){$7$};
\end{scope}	

\begin{scope}[shift={(17,-10)}]	
	\draw (0,0)--(3,0);
	\draw (0,-1)--(3,-1);
	\draw (0,-2)--(3,-2);	
	\draw (0,0)--(0,-2);
	\draw (1,0)--(1,-2);
	\draw (2,0)--(2,-2);
	\draw (3,0)--(3,-2);
\node at (1.5, -1.5){$8$};
\node at (2.5, -1.5){$8$};
\end{scope}	

\begin{scope}[shift={(21,-10)}]	
	\draw (0,0)--(3,0);
	\draw (0,-1)--(3,-1);
	\draw (0,-2)--(1,-2);	
	\draw (0,0)--(0,-2);
	\draw (1,0)--(1,-2);
	\draw (2,0)--(2,-1);
	\draw (3,0)--(3,-1);
\node at (2.5, -0.5){$9$};
\node at (0.5,-1.5){$9$};
\end{scope}	

\begin{scope}[shift={(25,-10)}]	
	\draw (0,0)--(2,0);
	\draw (0,-1)--(2,-1);
	\draw (0,0)--(0,-1);
	\draw (1,0)--(1,-1);
	\draw (2,0)--(2,-1);
\node at (0.5,-0.5){$10$};
\node at (1.5, -0.5){$10$};
\end{scope}	
	\end{tikzpicture}
\end{figure}
We then assemble these skew Young tableaux to form the reverse semi-standard Young tableau \( T_{\overrightarrow{x}} \) as follows:
\begin{figure}[H] 
	\begin{tikzpicture}[scale=0.48]
	\draw (0,0)--(8,0);
	\draw (0,-1)--(8,-1);
	\draw (0,-2)--(6,-2);
	\draw (0,-3)--(4,-3);
	\draw (0,-4)--(3,-4);
	\draw (0,-5)--(3,-5);
	\draw (0,-6)--(2,-6);
	\draw (0,-7)--(1,-7);	
	\draw (0,-8)--(1,-8);	
		
\draw (0,0)--(0,-8);
\draw (1,0)--(1,-8);
\draw (2,0)--(2,-6);
\draw (3,0)--(3,-5);
\draw (4,0)--(4,-3);
\draw (5,0)--(5,-2);
\draw (6,0)--(6,-2);
\draw (7,0)--(7,-1);
\draw (8,0)--(8,-1);

\node at (0.5,-0.5){$10$};
\node at (1.5, -0.5){$10$};
\node at (2.5, -0.5){$9$};
\node at (3.5, -0.5){$7$};
\node at (4.5, -0.5){$6$};
\node at (5.5, -0.5){$3$};
\node at (6.5, -0.5){$2$};
\node at (7.5, -0.5){$1$};
\node at (0.5,-1.5){$9$};
\node at (1.5, -1.5){$8$};
\node at (2.5, -1.5){$8$};
\node at (3.5, -1.5){$6$};
\node at (4.5, -1.5){$4$};
\node at (5.5, -1.5){$1$};
\node at (0.5,-2.5){$7$};
\node at (1.5, -2.5){$6$};
\node at (2.5, -2.5){$5$};
\node at (3.5, -2.5){$5$};
\node at (0.5,-3.5){$6$};
\node at (1.5, -3.5){$4$};
\node at (2.5, -3.5){$3$};
\node at (0.5,-4.5){$5$};
\node at (1.5,-4.5){$2$};
\node at (2.5,-4.5){$1$};
\node at (0.5,-5.5){$4$};
\node at (1.5,-5.5){$1$};
\node at (0.5,-6.5){$3$};
\node at (0.5,-7.5){$1$};
\end{tikzpicture}
\end{figure}
\end{Ex}

\begin{Thm}\label{thm: psilambda}
The map $\psi_\lambda$ in \eqref{eq:psi}  is a crystal isomorphism.
\end{Thm}
\begin{proof}
For a given $\lambda\in P^+$, it is straightforward to verify that $\psi_\lambda$ is a one-to-one map. We will first show that $\psi_\lambda$ preserves the weight map $\wt$.

\vskip 1mm

By the definition of $\wt$ in \eqref{eq:crystal operator}, we have 
\begin{equation}\label{eq:wtx=wtYx}
\begin{aligned}
&\wt(\overrightarrow{x})\\
=&\lambda-\sum_{1<i+j\leq n+1}x_j^{(i)}\alpha_{i_{(j-1)n+i}}
=\sum_{i=1}^n\lambda_i\epsilon_i-\sum_{i=1}^n(\sum_{j=1}^{n+1-i}x_j^{(i)}(\epsilon_i-\epsilon_{i+1}))\\
=&\sum_{i=1}^n(\lambda_i-\sum_{j=1}^{n+1-i}x_j^{(i)})\epsilon_i+\sum_{i=2}^{n+1}\sum_{j=1}^{n+2-i}x_j^{(i-1)}\epsilon_i\\
=&\sum_{i=1}^{n+1}((\lambda_i-\sum_{j=1}^{n+1-i}x_j^{(i)})+\sum_{j=1}^{n+2-i}x_j^{(i-1)})\epsilon_i\\
=&\sum_{i=1}^{n+1}|\Lambda_i(\overrightarrow{x})/\Lambda_{i+1}(\overrightarrow{x})|\epsilon_i=\wt(T_{\overrightarrow{x}}).
\end{aligned}
\end{equation}

Next, we will show that $\varepsilon_i(\overrightarrow{x})=\varepsilon_i(T_{\overrightarrow{x}})$. By Lemma \ref{lem:sigmaigeqsigma0}, it suffices to consider the case  $\sigma^{(i)}(\overrightarrow{x})\geq\sigma_0^{(i)}(\overrightarrow{x})$. In this case, we have $\varepsilon_i(\overrightarrow{x})=\sigma^{(i)}(\overrightarrow{x})$.

\vskip 1mm

By applying  the Middle-Eastern reading to the reverse Young tableau $T_{\overrightarrow{x}}$ (cf. \cite[Definition 7.3.4]{HK02}), we obtain the following sequence of boxes:
\begin{equation}\label{eq:MERT}
\underbrace{(b_1b_2\cdots b_{\lambda_1})}_{\text{$1$-th row}}
\underbrace{(b_{\lambda_1+1}\cdots b_{\lambda_1+\lambda_2})}_{\text{$2$-th row}}
\cdots 
\underbrace{(b_{\sum_{k=1}^{j-1}\lambda_{k}+1}\cdots b_{\sum_{k=1}^{j}\lambda_{k}})}_{\text{$j$-th row}}
\cdots
\underbrace{(b_{\sum_{k=1}^{n-1}\lambda_{k}+1}\cdots b_{\sum_{k=1}^{n}\lambda_{k}})}_{\text{$n$-th row}},
\end{equation}
where the subscript indicates the row of $T_{\overrightarrow{x}}$ in which the box $b_i$ is located.

\vskip 1mm

In the sequence of boxes \eqref{eq:MERT}, we  count  the number of $i$-boxes and $(i+1)$-boxes in each row, and apply the map $\phi$ to obtain the following sequences:
\begin{equation}\label{eq:sequence iandi+1}
\begin{aligned}
\underbrace{
\underbrace{|\Lambda_{i}^{(1)}/\Lambda_{i+1}^{(1)}|}_{i\text{-box}},\ 
\underbrace{|\Lambda_{i+1}^{(1)}/\Lambda_{i+2}^{(1)}|}_{(i+1)\text{-box}}}_{\text{$1$-th row}}
\cdots
&\underbrace{
\underbrace{|\Lambda_{i}^{(j)}/\Lambda_{i+1}^{(j)}|}_{i\text{-box}},\ 
\underbrace{|\Lambda_{i+1}^{(j)}/\Lambda_{i+2}^{(j)}|}_{(i+1)\text{-box}}}_{\text{$j$-th row}}
\cdots
\underbrace{
\underbrace{|\Lambda_{i}^{(n)}/\Lambda_{i+1}^{(n)}|}_{i\text{-box}},\ 
\underbrace{|\Lambda_{i+1}^{(n)}/\Lambda_{i+2}^{(n)}|}_{(i+1)\text{-box}}}_{\text{$n$-th row}}\\
&\qquad\qquad\ \ \Big\downarrow\phi\\
\underbrace{
\underbrace{|\Lambda_{i}^{(1)}/\Lambda_{i+1}^{(1)}|}_{(n+2-i)\text{-box}},\ 
\underbrace{|\Lambda_{i+1}^{(1)}/\Lambda_{i+2}^{(1)}|}_{(n+1-i)\text{-box}}}_{\text{$1$-th row}}
\cdots
&\underbrace{
\underbrace{|\Lambda_{i}^{(j)}/\Lambda_{i+1}^{(j)}|}_{(n+2-i)\text{-box}},\ 
\underbrace{|\Lambda_{i+1}^{(j)}/\Lambda_{i+2}^{(j)}|}_{(n+1-i)\text{-box}}}_{\text{$j$-th row}}
\cdots
\underbrace{
\underbrace{|\Lambda_{i}^{(n)}/\Lambda_{i+1}^{(n)}|}_{(n+2-i)\text{-box}},\ 
\underbrace{|\Lambda_{i+1}^{(n)}/\Lambda_{i+2}^{(n)}|}_{(n+1-i)\text{-box}}}_{\text{$n$-th row}}
\end{aligned}
\end{equation}

By \cite[Theorem 7.3.6]{HK02}, the Far-Eastern and Middle-Eastern readings of Young tableaux induce the same crystal structure. Therefore, based on the number of $(n+2-i)$-blocks and $(n+1-i)$-blocks in
the second sequence of \eqref{eq:sequence iandi+1}, we conclude that
\begin{equation}\label{eq:phiphiTx}
\varphi_{n+1-i}(\phi(T_{\overrightarrow{x}}))=\max_{1\leq j\leq n}\{\sum_{k=j}^n(|\Lambda_{i+1}^{(k)}/\Lambda_{i+2}^{(k)}|-|\Lambda_{i}^{(k+1)}/\Lambda_{i+1}^{(k+1)}|)\}
\end{equation}

By Lemma \ref{lem:sigmaixvalue}, we obtain
	\begin{align*}
&\varepsilon_i(T_{\overrightarrow{x}})=\varphi_{n+1-i}(\phi(T_{\overrightarrow{x}}))
=\max_{1\leq j\leq n}\{\Lambda_{i+1}^{(j)}+2\sum_{k=j+1}^n\Lambda_{i+1}^{(k)}-\sum_{k=j}^{n-1}\Lambda_{i+2}^{(k)}-\sum_{k=j}^{n-1}\Lambda_i^{(k+1)}\}\\
=&\max_{1\leq j\leq n}\{\Lambda_{i+1}^{(j)}+2\sum_{k=j+1}^{\min(n,n+1-i)}\Lambda_{i+1}^{(k)}-\sum_{k=j}^{\min(n-1,n-i)}\Lambda_{i+2}^{(k)}-\sum_{k=j}^{\min(n-1,n+1-i)}\Lambda_i^{(k+1)}\}\\
=&\max_{1\leq j\leq n}\{\Lambda_{i+1}^{(j)}+2\sum_{k=j+1}^{n+1-i}\Lambda_{i+1}^{(k)}-\sum_{k=j}^{n-i}\Lambda_{i+2}^{(k)}-\sum_{k=j}^{\min(n-1,n+1-i)}\Lambda_i^{(k+1)}\}\\
=&\max_{1\leq j\leq n}\{x_j^{(i)}+\lambda_{j+i}+2(\sum_{k=j+1}^{n+1-i}x_k^{(i)}+\lambda_{k+i})-(\sum_{k=j}^{n-i}x_k^{(i+1)}+\lambda_{i+k+1})\\
&\qquad -(\sum_{k=j}^{\min(n-1,n+1-i)}x_{k+1}^{(i-1)}+\lambda_{i+k})\}\\
=&\max_{1\leq j\leq n}\{x_j^{(i)}+2\sum_{k=j+1}^{n+1-i}x_k^{(i)}-\sum_{k=j}^{n-i}x_k^{(i+1)}-\sum_{k=j}^{n+1-i}x_{k+1}^{(i-1)}\}\\
=&\max_{1\leq j\leq n}\{x_j^{(i)}-x_j^{(i+1)}+2\sum_{k=j+1}^{n+1-i}x_k^{(i)}-\sum_{k=j+1}^{n-i}x_k^{(i+1)}-\sum_{k=j+1}^{n+2-i}x_{k}^{(i-1)}\}
=\sigma^{(i)}(\overrightarrow{x})=\varepsilon_i(\overrightarrow{x}).
\end{align*}

\vskip 1mm

We now proceed to verify that  $T_{\tilde{e}_i\overrightarrow{x}}=\tilde{e}_iT_{\overrightarrow{x}}$ for all $i\in I$. 

\vskip 1mm

If $ \sigma^{(i)}(\overrightarrow{x})\leq 0$, then we obtain $\tilde{e}_i\overrightarrow{x}=\mathbf 0$ and 
 $$
 \varphi_{n+1-i}(\phi(T_{\overrightarrow{x}}))=\varepsilon_i(T_{\overrightarrow{x}})=\varepsilon_i(\overrightarrow{x})=\sigma^{(i)}(\overrightarrow{x})\leq 0.
 $$
Therefore, we have $\tilde{e}_iT_{\overrightarrow{x}}=\phi^{-1}(\tilde{f}_{n+1-i}\phi(T_{\overrightarrow{x}}))=\mathbf 0=T_{\tilde{e}_i\overrightarrow{x}}$.

\vskip 2mm

If  $ \sigma^{(i)}(\overrightarrow{x})>0$, we assume that
$\tilde{e}_i\overrightarrow{x}=\overrightarrow{x}-\overrightarrow{e_l}$.
By the definition of $\tilde{e}_i$ in \eqref{eq:crystal operator}, we obtain 
\begin{equation}\label{eq:l=max}
l=\max M^{(i)}(\overrightarrow{x})=\{k\mid i_k=i,\ \sigma_k(\overrightarrow{x})=\sigma^{(i)}(\overrightarrow{x})\}.
\end{equation}

It follows that $l=(v-1)n+i$ for some $v\geq 1$.  Then, we have 
\begin{equation}\label{eq:Lambdaex}
	\Lambda_j(\tilde{e}_i\overrightarrow{x})=
	\begin{cases}
		\Lambda_j(\overrightarrow{x})&\text{if}\ j\neq i+1,\\
		\Lambda_{j}(\overrightarrow{x})-(\delta_{v,k})_{1\leq k\leq n+1-i}&\text{if}\ j=i+1.
\end{cases}
\end{equation}

By applying the expression of $\Lambda_j(\tilde{e}_i\overrightarrow{x})$ in  \eqref{eq:Lambdaex}, we deduce that
\begin{equation}\label{eq:echange}
|T_{\tilde{e}_i\overrightarrow{x}}|_{i+1}^{(v)}-|T_{\overrightarrow{x}}|_{i+1}^{(v)}=-1,\quad |T_{\tilde{e}_i\overrightarrow{x}}|_{i}^{(v)}-|T_{\overrightarrow{x}}|_{i}^{(v)}=1,\quad
|T_{\tilde{e}_i\overrightarrow{x}}|_{m}^{(u)}-|T_{\overrightarrow{x}}|_{m}^{(u)}=0
\end{equation}	
for $m\neq i, i+1$.

\vskip 1mm

By the formula in \eqref{eq:phiphiTx}, we obtain
\begin{equation*}
\sigma_{(j-1)n+i}(\overrightarrow{x})=\sum_{k=j}^n(|\Lambda_{i+1}^{(k)}/\Lambda_{i+2}^{(k)}|-|\Lambda_{i}^{(k+1)}/\Lambda_{i+1}^{(k+1)}|).
\end{equation*}

From \eqref{eq:l=max}, it follows that \( v \) is the maximum number satisfying

\begin{equation*}
\sum_{k=v}^n(|\Lambda_{i+1}^{(k)}/\Lambda_{i+2}^{(k)}|-|\Lambda_{i}^{(k+1)}/\Lambda_{i+1}^{(k+1)}|)=\varphi_{n+1-i}(\phi(T_{\overrightarrow{x}})).
\end{equation*}

Therefore, the operator \( \tilde{f}_{n+1-i} \) acts on \( \phi(T_{\overrightarrow{x}}) \) by replacing the entry \( n+1-i \) with \( n+2-i \) in the \( v \)-th row. This implies that the operator \( \tilde{e}_i \) acts on \( T_{\overrightarrow{x}} \) by replacing the entry \( i+1 \) with \( i \) in the  \( v \)-th row. By \eqref{eq:echange}, we conclude that \( \tilde{e}_i T_{\overrightarrow{x}} = T_{\tilde{e}_i \overrightarrow{x}} \).

\vskip 1mm

Similarly, one can show that \( \tilde{f}_i T_{\overrightarrow{x}} = T_{\tilde{f}_i \overrightarrow{x}} \) by applying the same argument.
\end{proof}

\section{Crystal structure on the set of Gelfand–Tsetlin patterns}\label{sec:CGTP}

Let $\lambda=\sum_{i=1}^n\lambda_i\epsilon_i\in P^+$. We set $\lambda_{n+1}=0$ and $y_k^{(0)}=\lambda_k$ for all $1\leq k\leq n+1$.

\vskip 1mm

We now consider the following set of nonnegative integers:
\begin{equation}\label{eq: GT}
\mathcal {GT}_{\lambda}=\{y_k^{(i)}\in\mathbb Z_{\geq 0}\mid y_k^{(i-1)}\geq y_k^{(i)}\geq y_{k+1}^{(i-1)}, \ 1\leq i\leq n,\ 1\leq k\leq n+1-i \}.
\end{equation}
The elements of $\mathcal {GT}_{\lambda}$ correspond precisely to the Gelfand–Tsetlin patterns and are characterized by the following system of integers

\begin{equation}\label{eq:gtpattern}
\begin{aligned}
\lambda_1&\ \qquad\ \ \  \lambda_2&&\qquad \lambda_3&& \cdots\ \ \lambda_k&&\qquad \lambda_{k+1}&&\ \cdots\ \ \lambda_{n-1}&&\qquad \lambda_n && \qquad\! 0\\
&\ \ y_1^{(1)} &&y_2^{(1)} &&\! \!  y_3^{(1)}\ \ \cdots &&\ y_k^{(1)} && y_{k+1}^{(1)}\ \ \cdots && y_{n-1}^{(1)} &&\! \!\! y_n^{(1)}\\
&\ \qquad\ \ y_1^{(2)} &&\quad\ \  y_2^{(2)} && \quad \ \ y_3^{(2)} && \cdots \quad\! y_k^{(2)} && \cdots \quad y_{n-2}^{(2)} &&\quad \ \ \ y_{n-1}^{(2)}\\
&\qquad\quad\quad\ \ \ \begin{rotate}{125}$\cdots$\end{rotate} &&  \empty &&\begin{rotate}{125}$\cdots$\end{rotate} &&  \begin{rotate}{125}$\cdots$\end{rotate} &&\!\!\!\!\!\!\!\!\!\! \!\!\!\!  \vdots && \!\!\!\!\!\!\!\!\!\! \!\!\!\!\!\! \!\!\!\! \begin{rotate}{45}$\cdots$\end{rotate} && \!\!\!\!\!\!\!\!\!\! \!\!\!\!\!\! \!\!\!\! \begin{rotate}{45}$\cdots$\end{rotate}\\
&\qquad\quad\quad\ \begin{rotate}{125}$\empty$\end{rotate} &&  \empty &&\begin{rotate}{125}$\empty$\end{rotate} &&\!\!\!\!\!\!\!\!\!\! \!  y_1^{(n-1)}&&\!\!\!\!\!\!\!\!\!\! \!  y_2^{(n-1)}&& \!\!\!\!\!\!\!\!\!\! \!\!\!\!\!\! \!\!\!\! \begin{rotate}{45}$\empty$\end{rotate} && \!\!\!\!\!\!\!\!\!\! \!\!\!\!\!\! \!\!\!\! \begin{rotate}{45}$\empty$\end{rotate}\\
&\qquad\quad\quad\ \begin{rotate}{125}$\empty$\end{rotate} &&  \empty &&\begin{rotate}{125}$\empty$\end{rotate} &&\  y_1^{(n)}&&\!\!\!\!\!\!\!\!\!\! \!  \empty&& \!\!\!\!\!\!\!\!\!\! \!\!\!\!\!\! \!\!\!\! \begin{rotate}{45}$\empty$\end{rotate} && \!\!\!\!\!\!\!\!\!\! \!\!\!\!\!\! \!\!\!\! \begin{rotate}{45}$\empty$\end{rotate}\\
\end{aligned}
\end{equation}
where the local configuration \( \begin{array}{cc} a & b \\ \multicolumn{2}{c}{c} \end{array} \)
is subject to the interlacing inequality \(a\ge c \ge b\).

\vskip 2mm

We define the map
\begin{equation}\label{eq:GTtoPoly}
\begin{aligned}
\varsigma_\lambda: \mathcal {GT}_\lambda\longrightarrow \Sigma_\iota[\lambda],\quad g\mapsto \overrightarrow{x}
\end{aligned}
\end{equation}
by setting 
\begin{equation}
x_k^{(i)}:=y_k^{(i)}-\lambda_{k+i} \text{ for $1\leq i\leq n$ and $1\leq k\leq n+1-i$}.
\end{equation}
Here, \( \{y_k^{(i)}\} \) are the entries of a Gelfand–Tsetlin pattern \( g \in \mathcal{GT}_\lambda \), and this map assigns to each pattern \( g \) an element \( \overrightarrow{x} \in \Sigma_\iota[\lambda] \).

\vskip 1mm

It follows from \eqref{eq: GT} that $\varsigma_\lambda$ is a one-to-one map.
Furthermore, it is easy to check that the composition map $\psi_\lambda\circ \varsigma_\lambda:
\mathcal {GT}_\lambda\stackrel{\varsigma_\lambda}{\longrightarrow} \Sigma_\iota[\lambda] \stackrel{\psi_\lambda}{\longrightarrow} \mathcal T'(\lambda)
$
is a bijection.
The inverse of $\psi_\lambda$ can be realized via Gelfand–Tsetlin patterns:

Given a RSSYT $T\in\mathcal T'(\lambda)$, let $g_T$ be its associated
Gelfand–Tsetlin pattern, then $\psi_\lambda^{-1}(T)=\varsigma_\lambda(g_T)$. The explicit formula is given by
\begin{equation*}
x_k^{(i)}:=(\sum_{j=i+1}^{n+2-k}|T|_j^{(k)})-\lambda_{k+i} \text{ for $1\leq i\leq n$ and $1\leq k\leq n+1-i$}.
\end{equation*}

\vskip 1mm

For any $g\in \mathcal{GT}_\lambda$, $i\in I$ and $j\in \{0\}\cup I$, we define 
\begin{align}
\sigma^{(i)}_j(g)&=y_j^{(i)}+2\sum_{k=j+1}^{n+1-i} y_k^{(i)}-\sum_{k=j}^{n-i}y_k^{(i+1)}-\sum_{k=j+1}^{n+2-i}y_k^{(i-1)},\label{eq:sigmaji}\\
M^{(i)}(g)&=\{j\mid \sigma_j^{(i)}(g)=\max_{j\in I}\{\sigma_j^{(i)}(g)\}\},
\label{eq:Mig}
\end{align}
where $y_0^{(i)}=0$ for all $i\in I$.

\begin{Thm}\label{thm:crystalGT}
The crystal structure on $\Sigma_\iota[\lambda]$ and $\mathcal T'(\lambda)$ induces a crystal structure on $\mathcal{GT}_\lambda$ as follows:

For any $g\in \mathcal{GT}_\lambda$,
\begin{equation}
\begin{aligned}
\mathrm{wt}(g)&=\sum_{i=1}^{n+1}\sum_{j=1}^{n+2-i}(y_j^{(i-1)}-y_j^{(i)})\epsilon_i,\\
\varepsilon_i(g)&=\max_{j\in I} \{
\sigma^{(i)}_j(g)\},\quad \varphi_i(g)=\langle\alpha_i^\vee,\mathrm{wt}(g) \rangle+\varepsilon_i(g),\\
\tilde{f}_ig&=
\begin{cases}
g+\delta_{v,\min M^{(i)}(g)}\overrightarrow{e}_{i+1,v}&\text{if} \ \max_{j\in I}\{\sigma_j^{(i)}(g)\}>\sigma_0^{(i)}(g),\\
\mathbf 0&\text{otherwise,}
\end{cases}
\\
\tilde{e}_ig&=
\begin{cases}
	g-\delta_{v,\max M^{(i)}(g)}\overrightarrow{e}_{i+1,v}&\text{if} \ \max_{j\in I}\{\sigma_j^{(i)}(g)\}>0,\\
\mathbf	0&\text{otherwise,}
\end{cases}\\
\end{aligned}
\end{equation}
where $g\pm\overrightarrow{e}_{i+1,v}$  denotes the Gelfand-Tsetlin pattern  obtained form $g$ by replacing $y^{(i)}_v$ with $y^{(i)}_v\pm 1$.
\end{Thm}
\begin{proof}
It is sufficient to prove that the following equalities hold
\begin{equation}\label{eq:wteef}
\begin{aligned}
\mathrm{wt}(\varsigma_\lambda(g))&=\mathrm{wt}(g), \quad \varepsilon_i(\varsigma_\lambda(g))=\varepsilon_i(g),\\ \tilde{e}_i(\varsigma_\lambda(g))&=\varsigma_\lambda(\tilde{e}_ig),\quad 
\tilde{f}_i(\varsigma_\lambda(g))=\varsigma_\lambda(\tilde{f}_ig).
\end{aligned}
\end{equation}

(1) From \eqref{eq:wtx=wtYx}, we obtain
\begin{align*}
&\mathrm{wt}(\varsigma_\lambda(g))\\
=&\sum_{i=1}^{n+1}(\lambda_i-\sum_{j=1}^{n+1-i}x_j^{(i)}+\sum_{j=1}^{n+2-i}x_j^{(i-1)})\epsilon_i\\
=&\sum_{i=1}^{n+1}(\lambda_i-\sum_{j=1}^{n+1-i}(y_j^{(i)}-\lambda_{i+j})+\sum_{j=1}^{n+2-i}(y_j^{(i-1)}-\lambda_{i+j-1}))\epsilon_i\\
=&\sum_{i=1}^{n+1}(\sum_{j=1}^{n+2-i}y_j^{(i-1)}-\sum_{j=1}^{n+1-i}y_j^{(i)})\epsilon_i=\mathrm{wt}(g).
\end{align*}

(2) By \eqref{eq:sigma0i} and Lemma \ref{lem:sigmaixvalue} , we have
\begin{align*}
&\varepsilon_i(\varsigma_\lambda(g))\\
=&\sigma^{(i)}(\overrightarrow{x})=\max_{1\leq j\leq n}\{x_j^{(i)}-x_j^{(i+1)}+2\sum_{k=j+1}^{n+1-i}x_k^{(i)}-\sum_{k=j+1}^{n-i}x_k^{(i+1)}-\sum_{k=j+1}^{n+2-i}x_{k}^{(i-1)}\}\\
=&\max_{1\leq j\leq n}\{
y_j^{(i)}-y_j^{(i+1)}+\lambda_{i+j+1}-\lambda_{i+j}+2\sum_{k=j+1}^{n+1-i}(y_k^{(i)}-\lambda_{k+i})\\
&\qquad -\sum_{k=j+1}^{n-i}(y_k^{(i+1)}-\lambda_{i+k+1})-\sum_{k=j+1}^{n+2-i}(y_k^{(i-1)}-\lambda_{k+i-1})\}\\
=&\max_{1\leq j\leq n} \{y_j^{(i)}-y_j^{(i+1)}+\lambda_{i+j+1}-\lambda_{i+j}+2\sum_{k=j+1}^{n+1-i}y_k^{(i)} -\sum_{k=j+1}^{n-i}y_k^{(i+1)}\\
&\qquad -\sum_{k=j+1}^{n+2-i}y_k^{(i-1)}+\sum_{k=j+2}^{n+1-i}\lambda_{i+k}+\sum_{k=j}^{n+1-i}\lambda_{k+i}-2\sum_{k=j+1}^{n+1-i}\lambda_{k+i}\}
\\
=&\max_{1\leq j\leq n} \{y_j^{(i)}+2\sum_{k=j+1}^{n+1-i}y_k^{(i)}-\sum_{k=j}^{n-i}y_k^{(i+1)}-\sum_{k=j+1}^{n+2-i}y_k^{(i-1)} \}=\varepsilon_i(g).
\end{align*}

(3) Based on (2), it follows that
\begin{equation}\label{eq:maxjIsigmajig}
\max_{j\in I}\{\sigma_j^{(i)}(g)\}=\sigma^{(i)}(\overrightarrow{x})=\sigma^{(i)}(\varsigma_\lambda(g)).
\end{equation}

By comparing the definitions of \( M^{(i)}(\overrightarrow{x}) \) in \eqref{eq:sigmai and Mi} and \( M^{(i)}(g) \) in \eqref{eq:Mig}, we can conclude that
\[
\max M^{(i)}(\overrightarrow{x}) = (\max M^{(i)}(g) - 1)n + i.
\]

If condition \eqref{eq:maxjIsigmajig} holds, then we have $\tilde{e}_ig=g-\overrightarrow{e}_{i+1,v}$, where $v=\max M^{(i)}(g)$.

This implies that $\tilde{e}_i$ transforms $y_v^{(i)}$ in $g$ to $y_v^{(i)}-1$, while leaving all other entries unchanged. Therefore, 
$\tilde{e}_i$ acts on $\varsigma_\lambda(g)$ by transforming $x_v^{(i)}$ to $x_v^{(i)}-1$, while preserving the values of all other components.
By \eqref{eq:crystal operator}, we obtain $\tilde{e}_i(\varsigma_\lambda(g))=\varsigma_\lambda(g)-\overrightarrow{e}_{(v-1)n+i}=\varsigma_\lambda(\tilde{e}_ig)$.

\vskip 1mm

The  conclusion $\tilde{f}_i(\varsigma_\lambda(g))=\varsigma_\lambda(\tilde{f}_ig)$ can be proved by a similar argument. Therefore, the formulas in \eqref{eq:wteef} holds, and the proof is complete.
\end{proof}

By Propositions \ref{prop:crystal-Yprime}, \ref{prop:rholambda} and
Theorems \ref{thm: psilambda}, \ref{thm:crystalGT}, we obtain the following corollary.

\begin{Cor}[Involution on the polyhedral model]
	Define
	\[
	\rho_\lambda^{\mathrm{poly}}
	:=\ \psi_\lambda^{-1}\ \circ\ \phi^{-1}\ \circ\ \evac\ \circ\ \phi\ \circ\ \psi_\lambda
	\ :\ \Sigma_\iota[\lambda]\ \longrightarrow\ \Sigma_\iota[\lambda],
	\]
	where $\evac(T)=\jdt\!\bigl((\phi^{-1}(T))^{\#}\bigr)$ is the Schützenberger involution on
	$\mathcal T(\lambda)$.
	Then $\rho_\lambda^{\mathrm{poly}}$ is an involution. Moreover,
	\[
	\rho_\lambda^{\mathrm{poly}}\circ \tilde f_i
	=\tilde e_{\,n+1-i}\circ \rho_\lambda^{\mathrm{poly}},\quad
	\wt\bigl(\rho_\lambda^{\mathrm{poly}}(\vec x)\bigr)
	=\tau\bigl(\wt(\vec x)\bigr).
	\]
\end{Cor}

In the Gelfand–Tsetlin pattern \eqref{eq:gtpattern}, we connect \( y_k^{(1)} \) and \( y_k^{(n+1-k)} \) with a line segment \( L_k \) for each \( 1 \leq k \leq n \). 
Thus, all the entries \( y_k^1, \ldots, y_k^{n+1-k} \) are located on the line segment \( L_k \). We also connect \( \lambda_i \) and \( \lambda_{n+1} \) with a line segment \( H_{i-1} \) for $2\leq i\leq n+1$.

\vskip 1mm

By Theorem \ref{thm:crystalGT}, we can combinatorially construct vectors in the polyhedral realization from the Gelfand–Tsetlin patterns as follows:

\vskip 1mm

For $g\in \mathcal{GT}_\lambda$, let \( \#L_i \) denote the sequence formed by arranging the numbers along the segment \( L_i \) from the upper left to the lower right, and let \( \#H_i \) 
represent the sequence formed by arranging the numbers along the segment  \( H_i \) from left to right.  For example, when \( n = 3 \), Figure \ref{HLn=3} below illustrates the positions of \( L_i \) and \( H_i \).
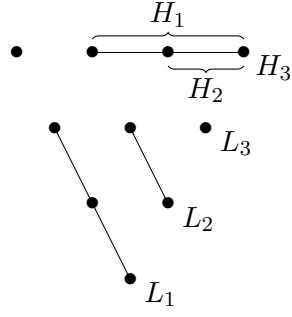
\begin{figure}[H] 
	\begin{tikzpicture}[scale=1]

			\node at (1,-1.5){$\bullet$};
			\node at (2, -1.5){$\bullet$};
			\node at (3, -1.5){$\bullet$};
			\node at (4, -1.5){$\bullet$};

			\node at (1.5,-2.5){$\bullet$};
			\node at (2.5, -2.5){$\bullet$};
			\node at (3.5, -2.5){$\bullet$};
			\node at (2,-3.5){$\bullet$};
			\node at (3, -3.5){$\bullet$};
			\node at (2.5,-4.5){$\bullet$};
			
\draw (1.5,-2.5)--(2.5,-4.5);	
	\node at (2.9,-4.7){$ L_1$};		
\draw (2.5,-2.5)--(3, -3.5);				
	\node at (3.4, -3.7){$\small L_2$};	
				\node at (3.9, -2.7){$\small L_3$};
\draw (2, -1.5)--(4, -1.5);	
\draw[decorate,decoration={brace,raise=5pt}] (2, -1.5)--(4, -1.5);	
	\node at (3, -1){$H_1$};	

\draw[decorate,decoration={brace,mirror,raise=5pt}] (3, -1.5)--(4, -1.5);	
	\node at (3.5, -2){$H_2$};
	
	\node at (4.4, -1.7){$H_3$};						
	\end{tikzpicture}
	\vskip -3mm
\caption{The positions of \( L_i \) and \( H_i \)}\label{HLn=3}	
\end{figure}

\vskip 1mm

It is straightforward to observe that the number of elements in \( \#L_i \) is equal to the number of elements in \( \#H_i \). Therefore, we define \( \#L_i-\#H_i \) to be the sequence obtained by subtracting \( \#H_i \) from \( \#L_i \) element-wise.
Then the vector $\varsigma_\lambda(g)$ is given by 
\begin{equation*}
(\cdots, 0, 0, \#L_n-\#H_n, \cdots, \#L_2-\#H_2,  \#L_1-\#H_1),
\end{equation*}
where the reading order of each part $\#L_i-\#H_i$ of  $\varsigma_\lambda(g)$ from right to left corresponds to the reading order of $\#L_i-\#H_i$ from top left to bottom right in the Gelfand–Tsetlin pattern.

\begin{Ex}
For fixed values of	$n$ and $\lambda$ as in Example \ref{ex:reverse}, the three models corresponding to the highest and lowest weight vectors are depicted in Figures  \ref{hwv}--\ref{lwv}.
\begin{figure}[H] 
	\begin{tikzpicture}[scale=0.5]
\begin{scope}[shift={(-2,0)}]
\node at (0.5,-0.5){$6$};
\node at (1.5, -0.5){$4$};
\node at (2.5, -0.5){$4$};
\node at (3.5, -0.5){$2$};
\node at (4.5, -0.5){$1$};
\node at (5.5, -0.5){$1$};
\node at (6.5, -0.5){$0$};
\node at (7.5, -0.5){$0$};
\node at (1,-1.5){$4$};
\node at (2, -1.5){$4$};
\node at (3, -1.5){$2$};
\node at (4, -1.5){$1$};
\node at (5, -1.5){$1$};
\node at (6, -1.5){$0$};
\node at (7, -1.5){$0$};
\node at (1.5,-2.5){$4$};
\node at (2.5, -2.5){$2$};
\node at (3.5, -2.5){$1$};
\node at (4.5, -2.5){$1$};
\node at (5.5, -2.5){$0$};
\node at (6.5, -2.5){$0$};
\node at (2,-3.5){$2$};
\node at (3, -3.5){$1$};
\node at (4, -3.5){$1$};
\node at (5, -3.5){$0$};
\node at (6, -3.5){$0$};
\node at (2.5,-4.5){$1$};
\node at (3.5,-4.5){$1$};
\node at (4.5,-4.5){$0$};
\node at (5.5,-4.5){$0$};
\node at (3,-5.5){$1$};
\node at (4,-5.5){$0$};
\node at (5,-5.5){$0$};
\node at (3.5,-6.5){$0$};
\node at (4.5,-6.5){$0$};
\node at (4,-7.5){$0$};

\node at (9,-3.5){$\varsigma_\lambda$};
\draw[thick,->] (8,-4)--(10,-4);
\end{scope}	

\begin{scope}[shift={(3,0)}]
\node at (6,-4) {$\overrightarrow{0}$};
\draw[thick,->] (7,-4)--(9,-4);
\node at (8,-3.5){$\psi_\lambda$};
\end{scope}	

\begin{scope}[shift={(12.5,-1)}]	
	\draw (0,0)--(6,0);
	\draw (0,-1)--(6,-1);
	\draw (0,-2)--(4,-2);
	\draw (0,-3)--(4,-3);
	\draw (0,-4)--(2,-4);
	\draw (0,-5)--(1,-5);
	\draw (0,-6)--(1,-6);
\draw (0,0)--(0,-6);
\draw (1,0)--(1,-6);
\draw (2,0)--(2,-4);
\draw (3,0)--(3,-3);
\draw (4,0)--(4,-3);
\draw (5,0)--(5,-1);
\draw (6,0)--(6,-1);
\node at (0.5,-0.5){$6$};
\node at (1.5, -0.5){$4$};
\node at (2.5, -0.5){$3$};
\node at (3.5, -0.5){$3$};
\node at (4.5, -0.5){$1$};
\node at (5.5, -0.5){$1$};
\node at (0.5,-1.5){$5$};
\node at (1.5, -1.5){$3$};
\node at (2.5, -1.5){$2$};
\node at (3.5, -1.5){$2$};
\node at (0.5,-2.5){$4$};
\node at (1.5, -2.5){$2$};
\node at (2.5, -2.5){$1$};
\node at (3.5, -2.5){$1$};
\node at (0.5,-3.5){$3$};
\node at (1.5, -3.5){$1$};
\node at (0.5,-4.5){$2$};
\node at (0.5,-5.5){$1$};
\end{scope}
\end{tikzpicture}
\vskip -3mm
\caption{Three models for the highest weight vector}\label{hwv}
\end{figure}
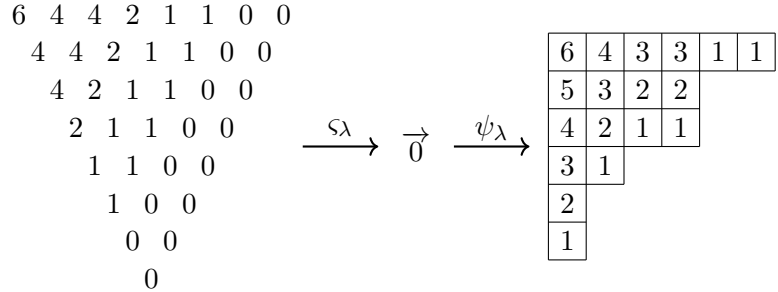
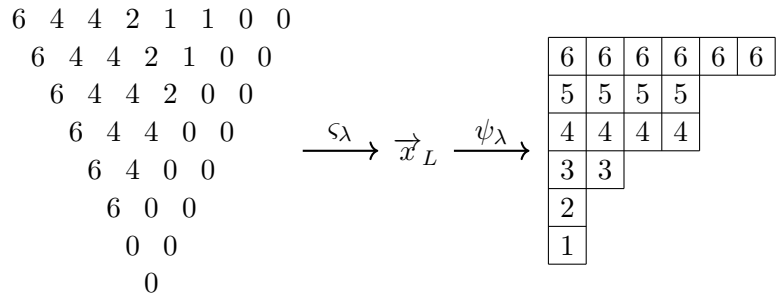
\begin{figure}[H] 
	\begin{tikzpicture}[scale=0.5]
	\begin{scope}[shift={(-2,0)}]
	
\node at (0.5,-0.5){$6$};
\node at (1.5, -0.5){$4$};
\node at (2.5, -0.5){$4$};
\node at (3.5, -0.5){$2$};
\node at (4.5, -0.5){$1$};
\node at (5.5, -0.5){$1$};
\node at (6.5, -0.5){$0$};
\node at (7.5, -0.5){$0$};
\node at (1,-1.5){$6$};
\node at (2, -1.5){$4$};
\node at (3, -1.5){$4$};
\node at (4, -1.5){$2$};
\node at (5, -1.5){$1$};
\node at (6, -1.5){$0$};
\node at (7, -1.5){$0$};
\node at (1.5,-2.5){$6$};
\node at (2.5, -2.5){$4$};
\node at (3.5, -2.5){$4$};
\node at (4.5, -2.5){$2$};
\node at (5.5, -2.5){$0$};
\node at (6.5, -2.5){$0$};
\node at (2,-3.5){$6$};
\node at (3, -3.5){$4$};
\node at (4, -3.5){$4$};
\node at (5, -3.5){$0$};
\node at (6, -3.5){$0$};
\node at (2.5,-4.5){$6$};
\node at (3.5,-4.5){$4$};
\node at (4.5,-4.5){$0$};
\node at (5.5,-4.5){$0$};
\node at (3,-5.5){$6$};
\node at (4,-5.5){$0$};
\node at (5,-5.5){$0$};
\node at (3.5,-6.5){$0$};
\node at (4.5,-6.5){$0$};
\node at (4,-7.5){$0$};
	
	\node at (9,-3.5){$\varsigma_\lambda$};
	\draw[thick,->] (8,-4)--(10,-4);
	\end{scope}	
	
	\begin{scope}[shift={(3,0)}]
	\node at (6,-4) {$\overrightarrow{x}_L$};
	\draw[thick,->] (7,-4)--(9,-4);
	\node at (8,-3.5){$\psi_\lambda$};
	\end{scope}	
	
	\begin{scope}[shift={(12.5,-1)}]	
	\draw (0,0)--(6,0);
	\draw (0,-1)--(6,-1);
	\draw (0,-2)--(4,-2);
	\draw (0,-3)--(4,-3);
	\draw (0,-4)--(2,-4);
	\draw (0,-5)--(1,-5);
	\draw (0,-6)--(1,-6);
	\draw (0,0)--(0,-6);
	\draw (1,0)--(1,-6);
	\draw (2,0)--(2,-4);
	\draw (3,0)--(3,-3);
	\draw (4,0)--(4,-3);
	\draw (5,0)--(5,-1);
	\draw (6,0)--(6,-1);
	\node at (0.5,-0.5){$6$};
	\node at (1.5, -0.5){$6$};
	\node at (2.5, -0.5){$6$};
	\node at (3.5, -0.5){$6$};
	\node at (4.5, -0.5){$6$};
	\node at (5.5, -0.5){$6$};
	\node at (0.5,-1.5){$5$};
	\node at (1.5, -1.5){$5$};
	\node at (2.5, -1.5){$5$};
	\node at (3.5, -1.5){$5$};
	\node at (0.5,-2.5){$4$};
	\node at (1.5, -2.5){$4$};
	\node at (2.5, -2.5){$4$};
	\node at (3.5, -2.5){$4$};
	\node at (0.5,-3.5){$3$};
	\node at (1.5, -3.5){$3$};
	\node at (0.5,-4.5){$2$};
	\node at (0.5,-5.5){$1$};
	\end{scope}
	\end{tikzpicture}
\vskip -3mm	
	\caption{Three models for the lowest weight vector}\label{lwv}
\end{figure}
Here, the vector $\overrightarrow{x}_L$ is given by 
\vskip -3mm
\begin{equation*}
(\cdots,0,1,1,0,0,3,3,2,0,0,3,3,2,0,0,0,5,5,4,2,2).
\end{equation*}
\end{Ex}

\section{Combinatorial description of the crystal embedding}\label{sec:crystal embedding}
In this section, we exploit  the correspondence between  polyhedral realizations and reverse tableau models to provide a combinatorial description of the crystal embedding 
$\mathcal B(\lambda){\hookrightarrow}\mathcal B(\infty)\otimes R_\lambda$.

\vskip 1mm

Let \( T \in \mathcal T'(\lambda) \), and let \( \xi_i^j(T) \) denote the number of columns between the leftmost \( i \)-box in the \( j \)-th row and the rightmost \( i \)-box in the \( (j+1) \)-th row of \( T \).

\begin{Lem}
Let \( g_T=(\Lambda_1,\Lambda_2,\cdots, \Lambda_{n+1}) \) denote the Gelfand–Tsetlin pattern corresponding to \( T \). 			
Then, for $1\leq i\leq n$ and $1\leq j\leq n+1-i$, we have $\xi_i^j(T)=\Lambda_{i+1}^j-\Lambda_{i}^{j+1}$.
\end{Lem}
\begin{proof}
As shown in  Figure \ref{xiij} below, the column index of the leftmost \( i \)-block in the \( j \)-th row is  \( \Lambda_{i+1}^j +1\), and the column index of the rightmost \( i \)-block in the \( (j+1) \)-th row is  \( \Lambda_i^{j+1} \).
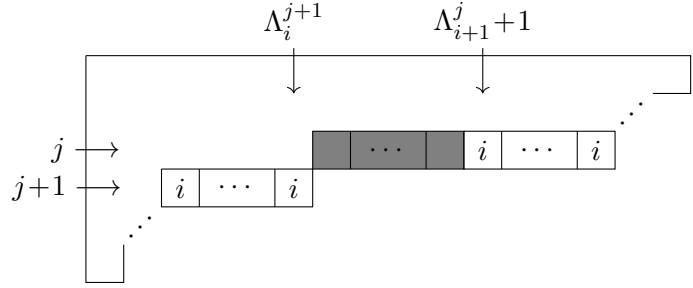
\begin{figure}[H] 
	\begin{tikzpicture}[scale=0.5]
	\draw (0,-1)--(16,-1)--(16, -2)--(15,-2);
\node[rotate=45] at (14.5,-2.5){$\cdots$};	
	\draw (0,-1)--(0,-7)--(1,-7)--(1,-6);
\node[rotate=45] at (1.5,-5.5){$\cdots$};

\begin{scope}[shift={(0,1)}]	
\draw (2,-5)--(6,-5)--(6,-6)--(2,-6)--(2,-5);
\draw (3,-5)--(3,-6);	
\node at (4, -5.5){$\cdots$};
\draw (5,-5)--(5,-6);
\node at (2.5, -5.5){$i$};
\node at (5.5, -5.5){$i$};	
	\node at (-0.5, -5.5){$j\!+\!1\longrightarrow$};	
			\node at (5.5, -1.1){$\Lambda_i^{j+1}$};
\draw[->] (5.5,-1.8)--(5.5,-3);			
\end{scope}	

\begin{scope}[shift={(4,2)}]	
	\draw[fill=gray,domain=0:6.28] (2,-5)--(6,-5)--(6,-6)--(2,-6)--(2,-5);
	\draw (3,-5)--(3,-6);	
	\node at (4, -5.5){$\cdots$};
	\draw (5,-5)--(5,-6);
	\node at (-4, -5.5){$j\longrightarrow$};	

\end{scope}	

\begin{scope}[shift={(8,2)}]	
	\draw (2,-5)--(6,-5)--(6,-6)--(2,-6)--(2,-5);
	\draw (3,-5)--(3,-6);	
	\node at (4, -5.5){$\cdots$};
	\draw (5,-5)--(5,-6);
	\node at (2.5, -5.5){$i$};
	\node at (5.5, -5.5){$i$};	
		\node at (2.5, -2.1){$\Lambda_{i+1}^j\!+\!1$};	
	\draw[->] (2.5,-2.8)--(2.5,-4);	
\end{scope}	

\end{tikzpicture}

\caption{Columns between the leftmost \( i \)-box in the \( j \)-th row and the rightmost \( i \)-box in the \( (j+1) \)-th row}\label{xiij}
\end{figure}

Therefore, the number of blocks in the shaded area is \( \Lambda_{i+1}^j-\Lambda_{i}^{j+1} \), which completes the proof of the lemma.
\end{proof}

We define 
\begin{equation}\label{eq:TlambdatoTinfty}
\mathtt{ml}:\mathcal T'(\lambda)\longrightarrow \mathcal T'(\infty),\quad
T\mapsto T^{\mathtt{ml}}
\end{equation}
by setting $z_{i}^j(T^{\mathtt{ml}}):=\xi_{j}^{n+2-i-j}(T)$ for $1\leq i\leq n$ and $1\leq j\leq n+1-i$.

\vskip 2mm

It is straightforward to verify that the map $\mathtt{ml}$ is injective. Furthermore, we have the following theorem:
\begin{Thm}\label{thm:tableauembedding}
The image of $\mathcal T'(\lambda)$	under the map $\mathtt{ml}$  is crystal isomorphic to $\mathcal B(\lambda)$.
\end{Thm}
\begin{proof}

For any $T\in \mathcal T'(\lambda)$ and its image $T^{\mathtt{ml}}$ under the map $\mathtt{ml}$, let \( \overrightarrow{x}_T =(\cdots,x_2,x_1)\) and \( \overrightarrow{b}_{T^{\mathtt{ml}}}=(\cdots,b_2,b_1) \) be the corresponding vectors in the polyhedral realizations of $\mathcal B(\lambda)$ and $\mathcal B(\infty)$, respectively.

\vskip 1mm

It follows from \eqref{eq:gtpattern} that $\Lambda_{i+1}^j=y_j^{(i)}$ for $1\leq i\leq n$ and $1\leq j\leq n+1-i$.
Based on the definition of $\psi_\infty$ in \eqref{eq:TinftytoSigmainfty}, together with that $\mathtt{ml}$  in \eqref{eq:TlambdatoTinfty},
we obtain
\begin{equation}\label{eq:aji=xji}
\begin{aligned}
b_j^{(i)}&=\sum_{k=1}^{i}z_{n+2-i-j}^k(T^{\mathtt{ml}})=\sum_{k=1}^{i}\xi_k^{i+j-k}(T)\\
&=\sum_{k=1}^{i}(\Lambda_{k+1}^{i+j-k}-\Lambda_k^{i+j-k+1})=\Lambda_{i+1}^j-\Lambda_1^{i+j}=\Lambda_{i+1}^j-\lambda_{i+j}=x_j^{(i)},
\end{aligned}
\end{equation}
which implies $\overrightarrow{x}_T=\overrightarrow{b}_{T^{\mathtt{ml}}}$.

\vskip 1mm

Thus, any vector in $\psi_\infty\circ \mathtt{ml}(\mathcal T'(\lambda))$ satisfies the conditions  in \eqref{eq:poly_Blambda}. Therefore, the image of \( \mathtt{ml} \) inherits the same crystal structure as \( \mathcal B(\lambda) \).
\end{proof}

For any $T\in\mathcal T'(\lambda)$, let $\xi_i^j:=\xi_i^j(T)$. Then, by Theorem  \ref{thm:tableauembedding}, the corresponding reverse marginally large tableau $T^{\mathtt{ml}}$ can be explicitly described by the following Figure \ref{fig:TmlfromT}.
\begin{figure}[H] 
	\begin{tikzpicture}[scale=0.5]
		\draw (-1,0)--(0,0);	
		\draw (-1,-1)--(0,-1);	
		\draw (-1,-2)--(0,-2);			
		\draw (-1,-4)--(0,-4);	
		\draw (-1,-5)--(0,-5);			
		\draw (-1,-6)--(0,-6);			
\node at (-0.5, -0.5){$\cdots$};			
	\node at (-0.5, -1.5){$\cdots$};
	\node at (-0.5, -3){$\vdots$};	
	\node at (-0.5, -4.5){$\cdots$};		
	\node at (-0.5, -5.5){$\cdots$};						
		\draw (0,0)--(29,0);
		\draw (0,-1)--(29,-1);
			\draw (0,-2)--(20,-2);
			\draw (0,-4)--(9,-4);	
		\draw (0,-5)--(9,-5);				
		\draw (0,-6)--(4,-6);	
		\draw (0,0)--(0,-6)	;			
		\draw (1,0)--(1,-6)	;
		\draw (2,0)--(2,-6)	;	
		\draw (4,-4)--(4,-6);	
		\draw (5,-5)--(5,-4);
	\draw (7,-5)--(7,-4);	
	\draw (9,-5)--(9,-4);	
	\draw (9,-5)--(9,-4);	
	\draw (11,0)--(11,-2);	
\draw (12,0)--(12,-2);	
\draw (14,-1)--(14,-2);		
	\draw (16,-1)--(16,-2);	
		\draw (18,-1)--(18,-2);	
		\draw (20,0)--(20,-2);		
		\draw (21,0)--(21,-1);
		\draw (23,0)--(23,-1);		
		\draw (25,0)--(25,-1);			
		\draw (27,0)--(27,-1);			
		\draw (29,0)--(29,-1);			
\node at (0.5, -0.5){\tiny $n\!\!+\!\!1$};	
\node at (1.5, -0.5){\tiny $n\!\!+\!\!1$};	
\node at (6.5, -0.5){$\cdots$};	
\node at (11.5, -0.5){\tiny $n\!\!+\!\!1$};	
\node at (16, -0.5){$\cdots$};	
\node at (20.5, -0.5){\tiny $n\!\!+\!\!1$};	
\node at (22, -0.5){\tiny $\xi_n^1$};	
\node at (24, -0.5){\tiny $\xi_{n-1}^2$};	
\node at (26, -0.5){$\cdots$};	
\node at (28, -0.5){\tiny $\xi_1^n$};	
\node at (0.5, -1.5){\tiny $n$};	
\node at (1.5, -1.5){\tiny $n$};	
\node at (6.5, -1.5){$\cdots$};	
\node at (11.5, -1.5){\tiny $n$};	
\node at (13, -1.5){\tiny $\xi_{n-1}^1$};	
\node at (15, -1.5){\tiny $\xi_{n-2}^2$};	
\node at (17, -1.5){$\cdots$};	
	\node at (19, -1.5){\tiny $\xi_1^{n-1}$};			
\node at (0.5, -3){$\vdots$};	
\node at (1.5, -3){$\vdots$};	
\node at (6.5, -3){$\cdots$};	
\node[rotate=45] at (10, -3){$\cdots$};	
\node at (0.5, -4.5){\tiny $3$};	
\node at (1.5, -4.5){\tiny $3$};	
\node at (3, -4.5){$\cdots$};
\node at (4.5, -4.5){\tiny $3$};	
\node at (6, -4.5){\tiny $\xi_2^1$};	
\node at (8, -4.5){\tiny $\xi_1^2$};
\node at (0.5, -5.5){\tiny $2$};	
\node at (1.5, -5.5){\tiny $2$};	
\node at (3, -5.5){\tiny $\xi_1^1$};											
	\end{tikzpicture}	
	\caption{The reverse marginally large tableau \( T^{\mathtt{ml}} \) constructed from the reverse tableau \( T \).
	}	\label{fig:TmlfromT}	
\end{figure}
Here, $\xi_i^j$ represents the number of $i$-boxes in the 
$(n+2-i-j)$-th row of $T^{\mathtt{ml}}$.

\begin{Ex}
Let $n=4$, and $\lambda=12\epsilon_1+10\epsilon_2+8\epsilon_3+3\epsilon_4$.
We consider the following RSSYT:
\begin{figure}[H] 
	\begin{tikzpicture}[scale=0.5]
\draw (0,0)--(12,0);	
\draw (0,-1)--(12,-1);			
\draw (0,-2)--(10,-2);			
\draw (0,-3)--(8,-3);			
\draw (0,-4)--(3,-4);		
\draw (0,0)--(0,-4);			
\draw (1,0)--(1,-4);		
\draw (2,0)--(2,-4);			
\draw (3,0)--(3,-4);		
\draw (4,0)--(4,-3);		
\draw (5,0)--(5,-3);
\draw (6,0)--(6,-3);
\draw (7,0)--(7,-3);	
\draw (8,0)--(8,-3);		
\draw (9,0)--(9,-2);
\draw (10,0)--(10,-2);			
\draw (11,0)--(11,-1);
\draw (12,0)--(12,-1);			
\node at (0.5, -0.5){$5$};			
\node at (1.5, -0.5){$5$};				
\node at (2.5, -0.5){$5$};		
\node at (3.5, -0.5){$5$};	
\node at (4.5, -0.5){$5$};	
\node at (5.5, -0.5){$5$};	
\node at (6.5, -0.5){$5$};	
\node at (7.5, -0.5){$5$};	
\node at (8.5, -0.5){$4$};	
\node at (9.5, -0.5){$4$};	
\node at (10.5, -0.5){$3$};	
\node at (11.5, -0.5){$2$};	
\node at (0.5, -1.5){$4$};			
\node at (1.5, -1.5){$4$};				
\node at (2.5, -1.5){$4$};		
\node at (3.5, -1.5){$4$};	
\node at (4.5, -1.5){$4$};	
\node at (5.5, -1.5){$4$};	
\node at (6.5, -1.5){$3$};	
\node at (7.5, -1.5){$2$};	
\node at (8.5, -1.5){$2$};	
\node at (9.5, -1.5){$1$};	
\node at (0.5, -2.5){$3$};			
\node at (1.5, -2.5){$3$};				
\node at (2.5, -2.5){$3$};		
\node at (3.5, -2.5){$2$};	
\node at (4.5, -2.5){$2$};	
\node at (5.5, -2.5){$2$};	
\node at (6.5, -2.5){$1$};	
\node at (7.5, -2.5){$1$};	
\node at (0.5, -3.5){$2$};			
\node at (1.5, -3.5){$2$};				
\node at (2.5, -3.5){$1$};																
	\end{tikzpicture}			
\end{figure}

Then the corresponding RMLT is given by
\begin{figure}[H] 
	\begin{tikzpicture}[scale=0.5]
		\draw (-1,0)--(24,0);	
		\draw (-1,-1)--(24,-1);	
		\draw (-1,-2)--(15,-2);			
		\draw (-1,-3)--(7,-3);				
		\draw (-1,-4)--(3,-4);				
		\draw (0,0)--(0,-4);			
		\draw (1,0)--(1,-4);		
		\draw (2,0)--(2,-4);		
		\draw (3,0)--(3,-4);			
		\draw (4,0)--(4,-3);			
		\draw (5,0)--(5,-3);		
		\draw (6,0)--(6,-3);			
		\draw (7,0)--(7,-3);		
		\draw (8,0)--(8,-2);		
	\draw (9,0)--(9,-2);	
	\draw (10,0)--(10,-2);			
	\draw (11,0)--(11,-2);			
	\draw (12,0)--(12,-2);	
	\draw (13,0)--(13,-2);	
\draw (14,0)--(14,-2);	
	\draw (15,0)--(15,-2);	
	\draw (16,0)--(16,-1);		
	\draw (17,0)--(17,-1);
	\draw (18,0)--(18,-1);	
		\draw (19,0)--(19,-1);	
		\draw (20,0)--(20,-1);		
		\draw (21,0)--(21,-1);		
\draw (22,0)--(22,-1);		
	\draw (23,0)--(23,-1);	
	\draw (24,0)--(24,-1);		
	\node at (-0.5, -0.5){$\cdots$};			
	\node at (-0.5, -1.5){$\cdots$};
		\node at (-0.5, -2.5){$\cdots$};
		\node at (-0.5, -3.5){$\cdots$};	
	
\node at (0.5, -0.5){$5$};			
\node at (1.5, -0.5){$5$};				
\node at (2.5, -0.5){$5$};		
\node at (3.5, -0.5){$5$};	
\node at (4.5, -0.5){$5$};	
\node at (5.5, -0.5){$5$};	
\node at (6.5, -0.5){$5$};	
\node at (7.5, -0.5){$5$};	
\node at (8.5, -0.5){$5$};	
\node at (9.5, -0.5){$5$};	
\node at (10.5, -0.5){$5$};	
\node at (11.5, -0.5){$5$};	
\node at (12.5, -0.5){$5$};	
\node at (13.5, -0.5){$5$};	
\node at (14.5, -0.5){$5$};	
\node at (15.5, -0.5){$5$};	
\node at (16.5, -0.5){$4$};	
\node at (17.5, -0.5){$4$};	
\node at (18.5, -0.5){$3$};	
\node at (19.5, -0.5){$3$};	
\node at (20.5, -0.5){$3$};	
\node at (21.5, -0.5){$2$};	
\node at (22.5, -0.5){$1$};	
\node at (23.5, -0.5){$1$};	
\node at (0.5, -1.5){$4$};			
\node at (1.5, -1.5){$4$};				
\node at (2.5, -1.5){$4$};		
\node at (3.5, -1.5){$4$};	
\node at (4.5, -1.5){$4$};	
\node at (5.5, -1.5){$4$};	
\node at (6.5, -1.5){$4$};	
\node at (7.5, -1.5){$4$};	
\node at (8.5, -1.5){$3$};	
\node at (9.5, -1.5){$3$};	
\node at (10.5, -1.5){$3$};	
\node at (11.5, -1.5){$2$};	
\node at (12.5, -1.5){$1$};	
\node at (13.5, -1.5){$1$};	
\node at (14.5, -1.5){$1$};	
\node at (0.5, -2.5){$3$};			
\node at (1.5, -2.5){$3$};				
\node at (2.5, -2.5){$3$};
\node at (3.5, -2.5){$3$};		
\node at (4.5, -2.5){$2$};	
\node at (5.5, -2.5){$2$};	
\node at (6.5, -2.5){$1$};	
\node at (0.5, -3.5){$2$};			
\node at (1.5, -3.5){$1$};				
\node at (2.5, -3.5){$1$};					
	\end{tikzpicture}			
\end{figure}
	\end{Ex}

\begin{Cor}\label{cor:imageml}
	The image of the map $\mathtt{ml}$ is given by
	\[
	\mathtt{ml}\bigl(\mathcal T'(\lambda)\bigr)
	=\Bigl\{\,T\in \mathcal T'(\infty)\ \Big|\ 
	\sum_{k=1}^{\,i-j+1}z_{\,n+1-i}^{\,k}(T)\;-\;\sum_{k=1}^{\,i-j}z_{\,n+2-i}^{\,k}(T)\ \le\ \lambda_i-\lambda_{i+1},\
	1\le j\le i\le n\Bigr\},
	\]
	where the empty sum is understood to be $0$.
\end{Cor}

\begin{proof}
	For $T\in\mathcal T'(\infty)$, let $\vec x=\psi_\infty(T)$ so that, by
	\eqref{eq:TinftytoSigmainfty},
	\[
	x_j^{(i)}=\sum_{k=1}^{\,i} z_{\,n+2-i-j}^{\,k}(T)\qquad(1\le i\le n,\ 1\le j\le n+1-i).
	\]
	Fix $1\le j\le i\le n$. A direct substitution gives the (correct) difference identity
	\begin{equation}\label{eq:difference-z}
		x_j^{(i-j+1)}-x_j^{(i-j)}
		=\sum_{k=1}^{\,i-j+1} z_{\,n+1-i}^{\,k}(T)\;-\;\sum_{k=1}^{\,i-j} z_{\,n+2-i}^{\,k}(T).
	\end{equation}
	
	\smallskip
	\noindent\emph{($\subseteq$)} Suppose $T\in\mathtt{ml}(\mathcal T'(\lambda))$. Then there exists
	$S\in\mathcal T'(\lambda)$ with $T=\mathtt{ml}(S)$. By \eqref{eq:aji=xji},
	\[
	\psi_\infty(T)=\psi_\infty\bigl(\mathtt{ml}(S)\bigr)=\psi_\lambda(S)\in\Sigma_\iota[\lambda].
	\]
	Hence, by Theorem~\ref{thm:polyhedral A},
	\[
	x_j^{(i-j+1)}-x_j^{(i-j)}\ \le\ \lambda_i-\lambda_{i+1}\qquad(1\le j\le i\le n).
	\]
	Using \eqref{eq:difference-z} we obtain exactly the displayed inequalities in the statement.
	
	\smallskip
	\noindent\emph{($\supseteq$)} Conversely, let $T\in\mathcal T'(\infty)$ satisfy the inequalities in the statement. Then \eqref{eq:difference-z} implies
	\[
	x_j^{(i-j+1)}-x_j^{(i-j)}\ \le\ \lambda_i-\lambda_{i+1}\qquad(1\le j\le i\le n),
	\]
	so $\vec x=\psi_\infty(T)\in\Sigma_\iota[\lambda]$ by Theorem~\ref{thm:polyhedral A}. Therefore there exists a unique
	$S\in\mathcal T'(\lambda)$ with $\psi_\lambda(S)=\vec x$. Using \eqref{eq:aji=xji} again,
	\[
	\psi_\infty\bigl(\mathtt{ml}(S)\bigr)=\psi_\lambda(S)=\vec x=\psi_\infty(T).
	\]
	Since $\psi_\infty:\mathcal T'(\infty)\to\Sigma_\iota$ is a bijection (Proposition~\ref{prop:infcrystal}), we conclude $T=\mathtt{ml}(S)$. Hence $T$ lies in the image of $\mathtt{ml}$.
	
	Combining the two inclusions proves the corollary.
\end{proof}

In \cite{Lee14}, Lee  proposed  a definition of reverse marginally large tableaux that differs from the one adopted in this paper and presented two distinct realizations of  $\mathcal B(\lambda)$ via marginally large tableaux..
We now verify that the realization of $\mathcal B(\lambda)$ in Theorem \ref{thm:tableauembedding} coincides with that given by Lee. 

\vskip 1mm

Recall the definition of $r_{i,m}$$(i\in I, 1\leq m\leq n+1-i)$ in the reverse marginally large tableaux given in \cite[Section 6.1]{Lee14}, as well as the set $R(\infty)^\lambda$ consisting of restricted reverse marginally large tableaux.
We define a map $\theta_\lambda: \mathtt{ml}(\mathcal T'(\lambda))\to R(\infty)^\lambda$ as follows:

\vskip 1mm

For any $T^{\mathtt{ml}}\in \mathtt{ml}(\mathcal T'(\lambda))$, the value of $r_{i,m}$ in $\theta_\lambda(T^{\mathtt{ml}})$ is given by 
\begin{equation}\label{eq:rim}
r_{i,m}:=\sum_{k=1}^i z_{n+2-i-m}^k(T^{\mathtt{ml}}).
\end{equation}

\vskip 1mm

By \eqref{eq:rim}, we obtain $z_j^i(T^{\mathtt{ml}})=r_{i,n+2-i-j}-r_{i-1,n+3-i-j}$ for $i\in I$ and $1\leq j\leq n+1-i$. Moreover, by \ref{cor:imageml}, we have

\begin{align*}
&\sum_{k=1}^{i-j+1}(r_{k,n+2-(n+1-i+k)}-r_{k-1,n+3-(n+1-i+k)}) -\sum_{k=1}^{i-j}(r_{k,n+2-(n+2-i+k)}-r_{k-1,n+3-(n+2-i+k)})\\
&=r_{i-j+1,j}-r_{i-j,j}\leq
\lambda_i-\lambda_{i+1},
\end{align*}
which implies 
\begin{equation}\label{eq:rl-rl-1j}
r_{l,j}-r_{l-1,j}\leq \lambda_{l-1+j}-\lambda_{l+j} \text{ for } 1\leq l\leq n,\ 1\leq j\leq n+1-l.
\end{equation}

The conditions in \eqref{eq:rl-rl-1j} coincide with the conditions in \cite[(6.1)]{Lee14}. Therefore, the combinatorial description of the crystal embedding described in Theorem \ref{thm:tableauembedding} is consistent with the one constructed by Lee. 

\vskip 2mm

Recall the PBW basis \(\mathbf{B} := \mathbf{B}_{\mathbf{i}_0}\) of \(U_q^-(A_n)\) associated with the reduced word \(\mathbf{i}_0\) (cf.~\cite[Definition 3.2]{Saito11}).  
Since the elements in \(\mathbf{B}\) can be parametrized by the integer sequences in \(\mathbb{Z}_{\geq 0}^N\)\(( N = \frac{n(n+1)}{2})\),  
we identify \(\mathbb{Z}_{\geq 0}^N\) with \(\mathbf{B}\) whenever no ambiguity arises, and refer to the integer sequences of \( \mathbb Z_{\geq 0}^N \) corresponding to elements in \(\mathbf{B}\) as the \emph{Lusztig data}.  
The crystal structure on \(\mathbf{B}\) is described in \cite[Section 4.1]{Saito11}.

\vskip 1mm

It follows from \cite[Theorem 4.1.2]{Saito94} that $\mathcal B(\infty)\equiv\mathbf B\ q L(\infty)$. This yields a natural embedding $\iota_\lambda:\mathcal T'(\lambda)\hookrightarrow\mathbf B$. As an application of Theorem \ref{thm:tableauembedding}, we provide an explicit description of the crystal embedding $\iota_\lambda$ as follows:
\begin{Thm}
The Lusztig data of the embedding $\iota_\lambda:\mathcal T'(\lambda)\hookrightarrow\mathbf B$ is given by the  $N$-tuple of nonnegative integers  $\iota_\lambda(T)=(\xi_{n+2-l}^{l-k}(T))_{1\leq k<l\leq n+1}$ for each $T\in T'(\lambda)$.
\end{Thm}
\begin{proof}
According to the description of \( \mathbf B^{\mathbf i} \) given in \cite[Section 3.2]{Saito11}, every element in \(\mathbf B \) can be uniquely  determined by  a sequence of non-negative integers  \( (a_{k,l})_{1\leq k<l\leq n+1} \). 

\vskip 1mm

By comparing the crystal structure on $\mathbf B$  with the realization of 
\( \mathcal{B}(\infty) \) via marginally large tableaux, we obtain a bijection between \( \mathbf B \) and $\mathcal T(\infty)$ such that
\( a_{k,l}= y_k^l(T') \) for $T'\in \mathcal T(\infty)$.

\vskip 1mm

We consider the following composition of maps:
 \begin{equation*}
 \mathcal T'(\lambda)\stackrel{\mathtt{ml}}{\hookrightarrow} \mathcal T'(\infty)\stackrel{\eta}\to\mathcal T(\infty)\to\mathbf B.
 \end{equation*}
Then, for any $i\in I$, $1\leq j\leq n+1-i$ and $T\in\mathcal T'(\lambda)$, we have
\begin{equation*}
\xi_j^{n+2-i-j}(T)=z_{i}^j(T^{\mathtt{ml}})=y_i^{n+2-j}(\eta T^{\mathtt{ml}})=a_{i,n+2-j}.
\end{equation*}
By setting \( k=i \) and \( l = n+2-j \), we obtain the desired conclusion.
\end{proof}

We summarize all correspondences and normalizations in the following \emph{commuting} diagram:
\begin{figure}[H]
	\[
	\begin{tikzcd}[column sep=large, row sep=large]
		\mathcal T'(\lambda)
		\arrow[r, "\mathtt{ml}"]
		\arrow[d, "\psi_\lambda^{-1}"']
		& \mathcal T'(\infty)
		\arrow[r, "\eta"]
		\arrow[d, "\psi_\infty"']
		& \mathcal T(\infty)
		\arrow[d, "\chi", "\cong"']
		\\
		\Sigma_\iota[\lambda]
		\arrow[r, hookrightarrow]
		& \Sigma_\iota
		\arrow[r, "\mathrm{PL}"]
		& \mathbf B \;\cong\; \mathbb Z_{\ge 0}^{\,N}
	\end{tikzcd}
	\]
	\vspace{-3mm}
	\caption{Commuting diagram}\label{fig:commu}
\end{figure}
Here, $\chi:\mathcal T(\infty)\xrightarrow{\ \cong\ }\mathbf B$ identifies an MLT $T$ with PBW/Lusztig data by setting
$a_{k,\ell}:=y_k^{\,\ell}$ (for a fixed reduced word $\mathbf i_0$), and
$\mathrm{PL}:\Sigma_\iota\to\mathbf B$ is defined by the composition $\mathrm{PL}:=\chi\circ\eta\circ\psi_\infty^{-1}$.

All arrows in Figure~\ref{fig:commu} are strict crystal morphisms; the three vertical arrows are crystal isomorphisms, and the lower–left horizontal arrow is the natural inclusion of the highest–weight subcrystal.

\appendix 
\addtocontents{toc}{\protect\setcounter{tocdepth}{1}}
\section{The polyhedral realizations of crystal bases}\label{sec:poly_general}

We consider the $\mathbb Z$-lattice:
\begin{equation*}
	\mathbb Z^{\infty}:=\{(\cdots,x_k,\cdots,x_2,x_1)\mid x_k\in\mathbb Z\ \text{and}\ x_k=0 \ \text{for}\ k\gg 0\}.
\end{equation*}

Let $\iota=(\cdots,i_k,\cdots,i_2,i_1)$ be an infinite sequence such that $i_k\in I$ and 
\begin{equation*}
	i_k\neq i_{k+1},\quad \#\{k\mid i_k=i\}=\infty \ \text{for any}\ i\in I.
\end{equation*}

Given a fixed $\iota$, a crystal structure can be defined on $\mathbb Z^\infty$.
Let $\mathbb Z^\infty_\iota$ denote the corresponding crystal, which is described as follows:

\vskip 1mm

For any $\overrightarrow{x}=(\cdots,x_k,\cdots,x_2,x_1)\in \mathbb Z^\infty$, we define the following linear functions:
\begin{equation}\label{eq:funsigma}
	\begin{aligned}
		\sigma_k(\overrightarrow{x})&:=x_k+\sum_{j>k}\langle\alpha_{i_k}^\vee,\alpha_{i_j}\rangle x_j\quad \text{for}\ k\geq 1.
	\end{aligned}
\end{equation}
The condition $x_k=0$ for $k\gg 0$ ensures that the function $\sigma_k$ is well defined. We set
\begin{equation}\label{eq:sigmai and Mi}
	\begin{aligned}
		\sigma^{(i)}(\overrightarrow{x})&:=\max_{k:i_k=i}\sigma_k(\overrightarrow{x}),\\ M^{(i)}(\overrightarrow{x})&:=\{k\mid i_k=i,\ \sigma_k(\overrightarrow{x})=\sigma^{(i)}(\overrightarrow{x})\}.
	\end{aligned}
\end{equation}

We define $\wt : \mathbb Z_\iota^\infty \rightarrow P$, $\tilde{e}_i,
\tilde{f}_i: \mathbb Z_\iota^\infty \rightarrow \mathbb Z_\iota^\infty \cup \{ 0 \}$ and $\varepsilon_i,
\varphi_i : \mathbb Z_\iota^\infty \rightarrow \mathbb{Z} \cup \{-\infty\}$  $(i \in I)
$ as follows:
\begin{equation}\label{eq:crystal operator on Zinfty}
	\begin{aligned}
		\wt(\overrightarrow{x})&:=-\sum_{j=1}^\infty x_j\alpha_{i_j},
		\\		
		\tilde{f}_i\overrightarrow{x}&:=
		\overrightarrow{x}+\delta_{k,\min M^{(i)}(\overrightarrow{x})}\overrightarrow{e_k},
		\\
		\tilde{e}_i\overrightarrow{x}&:=
		\begin{cases}
			\overrightarrow{x}-\delta_{k,\max M^{(i)}(\overrightarrow{x})}\overrightarrow{e_k}&\text{if} \ \sigma^{(i)}(\overrightarrow{x})>0,\\
			\mathbf 0&\text{otherwise,}
		\end{cases}
		\\
		\varepsilon_i(\overrightarrow{x})&:=\sigma^{(i)}(\overrightarrow{x}),
		\\
		\varphi_i(\overrightarrow{x})&:=\langle\alpha_i^\vee,\wt(\overrightarrow{x})\rangle+\varepsilon_i(\overrightarrow{x}).
	\end{aligned}
\end{equation}

Note that the symbol $\mathbf 0$ in \eqref{eq:crystal operator on Zinfty} does not represent the zero vector $\overrightarrow{0}=(\cdots,0,0)$, but rather indicates a vector that is not within the connected component of the crystal graph.

\vskip 1mm

Let $\mathbb Z_{\geq 0}^\infty$ be the subset of $\mathbb Z^\infty$ consisting of tuples of non-negative integers.

\begin{Prop}\cite[Theorem 2.5]{NZ97}
	There exists a unique strict embedding of crystals
	\begin{equation*}
		\begin{aligned}
			\Psi_\iota:\mathcal B(\infty)&\hookrightarrow \mathbb Z_{\geq 0}^\infty\subset \mathbb Z_\iota^\infty,	\\
			u_\infty&\mapsto (\cdots,0,\cdots,0,0).
		\end{aligned}
	\end{equation*}
\end{Prop}

Recall the definition of the subset $\Sigma_\iota$ of $\mathbb Z_{\geq 0}^\infty$ in \cite[(3.6)]{NZ97}:
\begin{equation*}
	\Sigma_\iota:=\{\overrightarrow{x}\in\mathbb Z_{\geq 0}^\infty\mid \varphi(\overrightarrow{x})\geq 0\ \text{for any}\ \varphi\in \Xi_\iota\},
\end{equation*}
where the definition of  set $\Xi_\iota$ is given in \cite[(3.4)]{NZ97}.

\begin{Thm}\cite[Theorem 3.1]{NZ97}
	The image $\mathrm{Im}(\Psi_\iota)$ is equal to the set $\Sigma_\iota$.
\end{Thm}

\vskip 2mm

Let $\lambda\in P^+$. Recalling the definition of the crystal $R_\lambda$ from Example \ref{ex:Rlambda}, we consider the crystal $\mathbb Z_\iota^\infty[\lambda]:=\mathbb Z_\iota^\infty\otimes R_\lambda$. Since $R_\lambda$ consists of a single element, we can identify $\mathbb Z_\iota^\infty[\lambda]$ with $\mathbb Z^\infty$ as a set. We define the following linear functions:
\begin{equation}\label{eq:sigma0i}
	\begin{aligned}
		\sigma_0^{(i)}(\overrightarrow{x})&:=-\langle\alpha_i^\vee,\lambda\rangle+\sum_{j\geq 1}\langle\alpha_{i}^\vee,\alpha_{i_j}\rangle x_j\quad \text{for}\ i\in I.
	\end{aligned}
\end{equation}

Based on the definitions of $\sigma^{(i)}(\overrightarrow{x})$ and $M^{(i)}(\overrightarrow{x})$ in \eqref{eq:sigmai and Mi}, we define $\wt : \mathbb Z_\iota^\infty[\lambda] \rightarrow P$, $\tilde{e}_i,
\tilde{f}_i: \mathbb Z_\iota^\infty[\lambda] \rightarrow \mathbb Z_\iota^\infty[\lambda] \cup \{ 0 \}$ and $\varepsilon_i,
\varphi_i : \mathbb Z_\iota^\infty[\lambda] \rightarrow \mathbb{Z} \cup \{-\infty\}$  $(i \in I)$ as follows:
\begin{equation}\label{eq:crystal operator}
	\begin{aligned}
		\wt(\overrightarrow{x})&:=\lambda-\sum_{j=1}^\infty x_j\alpha_{i_j},
		\\		
		\tilde{f}_i\overrightarrow{x}&:=
		\begin{cases}
			\overrightarrow{x}+\delta_{k,\min M^{(i)}}\overrightarrow{e_k}&\text{if} \ \sigma^{(i)}(\overrightarrow{x})>\sigma_0^{(i)}(\overrightarrow{x}),\\
			\mathbf 0&\text{otherwise,}
		\end{cases}
		\\
		\tilde{e}_i\overrightarrow{x}&:=
		\begin{cases}
			\overrightarrow{x}-\delta_{k,\max M^{(i)}}\overrightarrow{e_k}&\text{if} \ \sigma^{(i)}(\overrightarrow{x})\geq\sigma_0^{(i)}(\overrightarrow{x})\ \text{and}\ \sigma^{(i)}(\overrightarrow{x})>0,\\
			\mathbf 0&\text{otherwise,}
		\end{cases}
		\\
		\varepsilon_i(\overrightarrow{x})&:=\max (\sigma^{(i)}(\overrightarrow{x}),\sigma_0^{(i)}(\overrightarrow{x})),
		\\
		\varphi_i(\overrightarrow{x})&:=\langle\alpha_i^\vee,\wt(\overrightarrow{x})\rangle+\varepsilon_i(\overrightarrow{x}).
	\end{aligned}
\end{equation}

\vskip 1mm

Recall the definition of the subset $\Sigma_\iota[\lambda]$ of $\mathbb Z_\iota^\infty[\lambda]$ in \cite[(4.14)]{Na99}:
\begin{equation*}
	\Sigma_\iota[\lambda]:=\{\overrightarrow{x}\in\mathbb Z_\iota^\infty[\lambda]\mid \varphi(\overrightarrow{x})\geq 0\ \text{for any}\ \varphi\in \Xi_\iota[\lambda]\},
\end{equation*}
where the definition of the  set $\Xi_\iota[\lambda]$ is given in \cite[(4.13)]{Na99}.
Then  the following theorem holds:
\begin{Thm}\cite[Theorem 3.2, Theorem 4.1]{Na99}
	\begin{enumerate}
		\item The map
		\begin{equation*}
			\Psi_\iota^{(\lambda)}:\mathcal B(\lambda)\stackrel{\Omega_\lambda}{\hookrightarrow}\mathcal B(\infty)\otimes R_\lambda\stackrel{\Psi_\iota\otimes \mathrm{id}}{\hookrightarrow} \mathbb Z_{\iota}^\infty\otimes R_\lambda=\mathbb Z_\iota^\infty[\lambda]
		\end{equation*}
		is the unique strict embedding of crystals such that $\Psi_\iota^{(\lambda)}(v_\lambda)=\overrightarrow{ 0}\otimes r_\lambda$, where $v_\lambda$ is the highest
		weight vector in $\mathcal B(\lambda)$.
		\item The set $\Sigma_\iota[\lambda]$ forms a subcrystal of 
		$\mathbb Z^\infty_\iota[\lambda]$ and coincides with the highest weight crystal $\mathcal B(\lambda)$.
	\end{enumerate}
	
\end{Thm}

\begin{Lem}\label{lem:sigmaixvalue}
	For any $\overrightarrow{x}\in\Sigma_\iota[\lambda]$, the value of $\sigma^{(i)}(\overrightarrow{x})$ is given by
	$$\max_{j\geq 1}\{x_j^{(i)}-x_j^{(i+1)}+2\sum_{k=j+1}^{n+1-i}x_k^{(i)}-\sum_{k=j+1}^{n+2-i}x_k^{(i-1)}-\sum_{k=j+1}^{n-i}x_k^{(i+1)}\}.$$
\end{Lem}
\begin{proof}
	By the definition of $\sigma^{(i)}(\overrightarrow{x})$ in \eqref{eq:sigmai and Mi} and the periodic sequence $\iota$ in \eqref{eq:periodic sequence}, it follows that
	\begin{equation}\label{eq:epsilonix}
		\begin{aligned}
			\sigma^{(i)}(\overrightarrow{x})&=\max_{k:i_k=i}\sigma_k(\overrightarrow{x})=\max_{j\geq 1}\sigma_{(j-1)n+i}(\overrightarrow{x})\\
			&=\max_{j\geq 1}\{x_{(j-1)n+i}+\sum_{l>(j-1)n+i}\langle\alpha_i^\vee, \alpha_{i_l}\rangle x_l\}\\
			&=\max_{j\geq 1}\{x_j^{(i)}-x_j^{(i+1)}+\sum_{k\geq j+1}(2x_k^{(i)}-x_k^{(i-1)}-x_k^{(i+1)})\}\\
			&=\max_{j\geq 1}\{x_j^{(i)}-x_j^{(i+1)}+2\sum_{k=j+1}^{n+1-i}x_k^{(i)}-\sum_{k=j+1}^{n+2-i}x_k^{(i-1)}-\sum_{k=j+1}^{n-i}x_k^{(i+1)}\},
		\end{aligned}
	\end{equation}
	as desired.
\end{proof}

\begin{Lem}\label{lem:sigmaigeqsigma0}
	For any $\overrightarrow{x}\in\Sigma_\iota[\lambda]$, we have $\sigma^{(i)}(\overrightarrow{x})\geq\sigma_0^{(i)}(\overrightarrow{x})$.
\end{Lem}
\begin{proof}

	If $\sigma^{(i)}(\overrightarrow{x})<\sigma^{(i)}_0(\overrightarrow{x})$, then by Lemma \ref{lem:sigmaixvalue}, we can assume that 
	$$
	\sigma^{(i)}(\overrightarrow{x})=x_j^{(i)}-x_j^{(i+1)}+2\sum_{k=j+1}^{n+1-i}x_k^{(i)}-\sum_{k=j+1}^{n+2-i}x_k^{(i-1)}-\sum_{k=j+1}^{n-i}x_k^{(i+1)}.
	$$
	Therefore, we obtain
	\begin{equation*} 
		\begin{aligned}
			&x_j^{(i)}-x_j^{(i+1)}+2\sum_{k=j+1}^{n+1-i}x_k^{(i)}-\sum_{k=j+1}^{n+2-i}x_k^{(i-1)}-\sum_{k=j+1}^{n-i}x_k^{(i+1)} \geq\\
			&
			x_1^{(i)}-x_1^{(i+1)}+2\sum_{k=2}^{n+1-i}x_k^{(i)}-\sum_{k=2}^{n+2-i}x_k^{(i-1)}-\sum_{k=2}^{n-i}x_k^{(i+1)},
		\end{aligned}
	\end{equation*}
	which implies
	\begin{equation}\label{eq:sigmigeq}
		\sum_{k=1}^jx_k^{(i)}+\sum_{k=2}^{(j-1)}x_k^{(i)}-\sum_{k=2}^{j}x_k^{(i-1)}-\sum_{k=1}^{j-1}x_k^{(i+1)}\leq 0.
	\end{equation}

	From the definition of $\sigma_0^{(i)}(\overrightarrow{x})$ in \eqref{eq:sigma0i}, it follows that
	\begin{align*}
		&\sigma_0^{(i)}(\overrightarrow{x})-\sigma^{(i)}(\overrightarrow{x})=\lambda_{i+1}-\lambda_i+\sum_{k=1}^j(2x_k^{(i)}-x_k^{(i-1)}-x_k^{(i+1)})-x_j^{(i)}+x_j^{(i+1)}\\
		=&\lambda_{i+1}-\lambda_i+\sum_{k=1}^jx_k^{(i)}+\sum_{k=1}^{j-1}x_k^{(i)}-\sum_{k=1}^jx_k^{(i-1)}-\sum_{k=1}^{j-1}x_k^{(i+1)}>0.
	\end{align*}
	
	By the condition in \eqref{eq:poly_Blambda}, we have $\lambda_i-\lambda_{i+1}\geq x_1^{(i)}-x_1^{(i-1)}$, thereby implying that
	\begin{align*}
		\sum_{k=2}^jx_k^{(i)}+\sum_{k=1}^{j-1}x_k^{(i)}-\sum_{k=2}^jx_k^{(i-1)}-\sum_{k=1}^{j-1}x_k^{(i+1)}>0.
	\end{align*}
	
	This inequality contradicts \eqref{eq:sigmigeq}. Thus $\sigma^{(i)}(\overrightarrow{x})<\sigma^{(i)}_0(\overrightarrow{x})$ is impossible. 
\end{proof}

\end{document}